\documentclass{amsart}

\usepackage{xcolor}

\usepackage[normalem]{ulem}

\usepackage[a4paper,top=3cm,bottom=2.5cm,
           left=2.5cm,right=2.5cm,marginparwidth=60pt]{geometry}
\usepackage[english]{babel} 
\usepackage[utf8]{inputenc}
\usepackage[T1]{fontenc} 
\usepackage{braket} 
\usepackage{tikz}
\usetikzlibrary{math,shapes.callouts} 
\usepackage{tikz-cd} 
\usepackage{enumerate}

\tikzcdset{
  cells={font=\everymath\expandafter{\the\everymath\displaystyle}},
}

\usepackage{amssymb} 

\usepackage[pdftex,colorlinks=true]{hyperref} 
\hypersetup{urlcolor=blue, citecolor=red, linkcolor=blue}

\usepackage{float} 
\usepackage{todonotes}

\usepackage[widespace]{fourier} 
\usepackage[capitalise]{cleveref}

\usepackage[symbol]{footmisc}

\DeclareMathOperator{\Spec}{Spec} 
\DeclareMathOperator{\Sets}{Sets} 
\DeclareMathOperator{\Sch}{Sch} 
\DeclareMathOperator{\Quot}{Quot}
\DeclareMathOperator{\Hilb}{Hilb}
 
\DeclareMathOperator{\Ann}{Ann} 

\DeclareMathOperator{\Span}{Span} 
\DeclareMathOperator{\Gr}{Gr}  
\DeclareMathOperator{\Coh}{Coh}  
 
\DeclareMathOperator{\supp}{Supp}

\DeclareMathOperator{\len}{length}

\DeclareMathOperator{\opp}{op}  
\DeclareMathOperator{\redu}{red}

\newcommand{\Calf}{\mathcal{F}} 
\newcommand{\Calo}{\mathcal{O}}

\newcommand{\Cali}{\mathcal{I}} 

\newcommand{\Calz}{\mathcal{Z}}
\newcommand{\Calk}{\mathcal{K}}
\newcommand{\C}{\mathbb{C}} 

\newcommand{\Z}{\mathbb{Z}} 
\newcommand{\A}{\mathbb{A}}
\newcommand{\N}{\mathbb{N}}
 
\newcommand{\mm}{\mathfrak{m}}

\DeclareMathOperator{\Hom}{Hom}

\newtheorem{theorem}{Theorem}[section]
\newtheorem{proposition}[theorem]{Proposition}
\newtheorem{introtheorem}{Theorem}[section] 
\makeatletter
\@addtoreset{introtheorem}{section}
\makeatother

\newtheorem{corollary}[theorem]{Corollary}
\newtheorem{lemma}[theorem]{Lemma}
\theoremstyle{definition}
\newtheorem{definition}[theorem]{Definition}
\newtheorem{notation}[theorem]{Notation}
\newtheorem*{convention}{Convention}

\theoremstyle{remark} 
\newtheorem{remark}[theorem]{Remark}
\newtheorem{example}[theorem]{Example}
\newtheorem*{conjecture*}{Conjecture}

\numberwithin{equation}{section}

\theoremstyle{plain} 
\newtheorem{conjecture}{Conjecture}
\newtheorem*{theorem*}{\textbf{Theorem}}

\title[Unexpected but recurrent phenomena for Quot and Hilbert schemes of points]{Unexpected but recurrent phenomena for\\ Quot and Hilbert schemes of points}

\author[F. Giovenzana]{Franco Giovenzana} \address[Franco Giovenzana]{Laboratoire de mathématiques d’Orsay, Université Paris Saclay, Rue Michel Magat, Bât. 307, 91405 Orsay, France}
\email{\href{mailto:franco.giovenzana@universite-paris-saclay.fr}{\tt franco.giovenzana@universite-paris-saclay.fr}}

\author[L. Giovenzana]{Luca Giovenzana} \address[Luca Giovenzana]{Department of Pure Mathematics\\ University of Sheffield\\ Hicks Building, Hounsfield Road\\ Sheffield, S3 7RH \\ United Kingdom}
\email{\href{mailto:l.giovenzana@sheffield.ac.uk}{\tt l.giovenzana@sheffield.ac.uk}}

\author[M. Graffeo]{Michele Graffeo} \address[Michele Graffeo]{Scuola Internazionale Superiore di Studi Avanzati (SISSA), Via Bonomea 265, 34136 Trieste, Italy}
\email{\href{mailto:mgraffeo@sissa.it}{\tt mgraffeo@sissa.it}}

\author[P. Lella]{Paolo Lella} \address[Paolo Lella]{Dipartimento di Matematica\\ Politecnico di Milano\\ Piazza Leonardo da Vinci 32\\ 20133 Milan\\ Italy}
\email{\href{mailto:paolo.lella@polimi.it}{\tt paolo.lella@polimi.it}}

\thanks{F.G.'s research is funded by Deutsche Forschungsgemeinschaft (DFG, German Research Foundation), Projektnummer 509501007 and partially supported by the European Research Council (ERC) under the European
Union’s Horizon 2020 research and innovation programme (ERC-2020-SyG-854361-HyperK). P.L.'s research is funded by the project PRIN 2020 \lq\lq Squarefree Gr\"obner degenerations, special varieties and related topics\rq\rq~(MUR, project number 2020355B8Y).
All the authors are members of the \lq\lq Gruppo Nazionale di Ricerca INdAM  GNSAGA\rq\rq}
 
\makeatletter
\@namedef{subjclassname@2020}{%
 \textup{2020} Mathematics Subject Classification}
\makeatother
\subjclass[2020]{
14C05, 
13D10, 
13F55 
13P10
}
\keywords{Quot scheme, Parity conjecture, Hilbert scheme of points, nested Hilbert scheme}

\begin{document}

\begin{abstract} 
We investigate some aspects of the geometry of two classical generalisations of the Hilbert schemes of points. Precisely, we show that parity conjecture for $\Quot_r^d\A^3$ already fails for $d=8$ and $r=2$ and that lots of the elementary components of the nested Hilbert schemes of points on smooth quasi-projective varieties of dimension at least 4 are generically non-reduced. We also deduce that nested Hilbert schemes of points on smooth surfaces have generically non-reduced components. Finally, we give an infinite family of elementary components of the classical Hilbert schemes of points. 
\end{abstract}
	
	\maketitle
 
\section*{Introduction}
Moduli spaces of sheaves are one of the most investigated objects in algebraic geometry. In the present note we are mainly interested in $\Quot$-schemes  and nested Hilbert schemes of points.

Given a quasi-projective variety $X$ defined over the field of complex numbers and two positive integers $r,d>0$,\hfill the\hfill $\Quot$-scheme\hfill $\Quot_r^d X$\hfill  parametrises\hfill isomorphism $\Calo_X$-classes\hfill of\hfill zero-dimensional\hfill quotients\\ $\Calo _{X}^r\twoheadrightarrow \mathcal F $  such that $\dim_\C H^0(X,\mathcal F)=d$. It is a quasi-projective scheme and it was introduced by Alexander  Grothendieck in \cite{FGA}.

When $r=1$ the points $[\Calf]\in \Quot_1^dX$ are understood as structure sheaves of closed zero-dimensional subschemes of $X$. In this special case the $\Quot$-scheme is called Hilbert scheme of $d$ points on $X$. We denote it by $\Hilb^{d} X$. This notion generalises to that of \emph{nested} Hilbert scheme of points as follows. If $\mathbf d\in\Z^r$ is a non-decreasing sequence of non-negative integers, the nested Hilbert scheme $\Hilb^{\mathbf d} X$ is the fine moduli space parametrising nestings $Z_1\subset \cdots \subset Z_r  $ of zero-dimensional closed subschemes of $X$ of respective length $ d_i$, for $i=1,\ldots,r$, where $\mathbf{d}=(d_1,\ldots,d_r)$.
We deal only with the case of a smooth variety $X$, and since the properties we address depend only locally on $X$, there is no harm in considering $X=\mathbb A^n_{\C}$. In this setting, the case $n=3$ is probably the most fascinating, being squashed between smooth Hilbert schemes and fairly well-understood $\Quot$-schemes for $n\leqslant 2$ and wildly singular Hilbert schemes for $n\geqslant 4$, \cite{joachim-pathologies,szachniewicz2021nonreducedness}. To give an idea of the lack of knowledge regarding the three-dimensional case, just think that the exact value of $d$ for which $\Hilb^{d} X$ and $\Quot^d_r X$ becomes reducible is not known, while this question has a complete answer in any dimension other than three \cite{8POINTS,JOACHIM,FPS,10points,Iarrob,IARRO,MAZZOLA,Klemen}.  For more details, we refer to \cite{JoachimQuestions} and therein references.  Similar problems are open for nested Hilbert schemes of points. For instance, when $\dim X=2$, the nested Hilbert scheme is irreducible for $r\leqslant 2$ and reducible in general for $r\geqslant5$. The cases $r=3,4$ are not yet understood \cite{IRRonenested,ALESSIONESTED}.

\begin{convention}
    Whenever not specified, we work over the field of complex numbers $\C$.  For a  quasi projective variety $X$, a fat point $Z\subset X$ is a closed zero-dimensional subscheme with support one point.
\end{convention}

\subsection*{Parity conjecture}
Motivated by expectations in enumerative geometry,  it was proven in \cite{PANDHARIPANDE} that as soon as $[Z]\in \Hilb^d \A^3$ is a monomial subscheme, the following formula is true 
\[
d\equiv {\dim_{\C}\mathsf T_{[Z]} \Hilb^d \A^3}\pmod2,
\]
		where $ \mathsf T_{[Z]}\Hilb^d \A^3$ denotes the tangent space to $\Hilb^d \A^3$ at $[Z]$. This result led Okounkov and Pandharipande to conjecture in \cite{CONGETTURA} that the same holds true at any point of the Hilbert scheme of points on a smooth threefold. More recently, a generalised version of the conjecture was posed by Ramkumar and Sammartano in \cite{SAMMARTANO}, where they proved its veridicity for all homogeneous $\Calo_{\A^3}$-modules of finite length. The conjecture, in its final form, is the following. 

\begin{conjecture} [Parity conjecture]\label{conjecture2}
		Let $r,d\in\N$ be two positive integers and let $X$ be an irreducible smooth threefold. Then, for any $[\mathcal F]\in \Quot_r^d X$, one has
		\[\dim_{\C} \mathsf T_{\mbox{\tiny $[\mathcal F]$}}\Quot_r^d X\equiv r\cdot d\  \pmod 2.\]
	\end{conjecture}
 \Cref{conjecture2} has been recently disproved by the authors by providing a ten-dimensional family of zero-dimensional schemes in $\Hilb^{12} \A^3$ with odd-dimensional tangent space {\cite[Corollary 3.2]{GGGL}}. This argument provides counterexamples to the parity conjecture for any $(r,d)$ with $r\geqslant1$ and $d\geqslant 12$.  
 In \Cref{sec:parity} we propose a variant of the counterexamples given in \cite{GGGL} showing  that, at the price of increasing the rank one has counterexamples of smaller length.  Our main result in this direction is the following. 
\begin{introtheorem}[{\Cref{thm:main}}]\label{thmintro:main}
    The parity conjecture (\Cref{conjecture2})  is false when $(r,d)$ runs in the following range:
    \begin{itemize}
        \item $r=1$ and $d\geqslant 12$,
        \item $r\geqslant 2$ and $d\geqslant8$.
        \end{itemize}
\end{introtheorem}

We remark that the counterexamples given in \cite{GGGL} by the authors prove the case $r=1 $ in \Cref{thmintro:main}. They are smoothable and the dimension of the tangent space to the Hilbert scheme at that points differs by 9 from the dimension of the smoothable component while the gap between the dimension of the principal component \cite{Jelisiejewsivic} and the dimension of the tangent space of the first counterexample in $\Quot_2^{8} \mathbb{A}^3$ is 5. In \Cref{subsec:explanation} we show that it is possible to construct a counterexample of rank 2 and length 8 starting from a counterexample of rank 1.
Following the general principle that the higher the rank $r$ is the more singular the $\Quot$-scheme is, we expect to find some other counterexample of length $d\leqslant 12$ and $r=1$ with a lower gap.
Recent results about the singularity of the Hilbert schemes announced at the workshop \lq\lq {The Geometry of Hilbert schmes of points\rq\rq\footnote{\href{https://sites.google.com/view/ghisp2024/home}{The Geometry of Hilbert schmes of points}, Levico Terme (TN), Italy, May 6-10, 2024.} suggest that there is not much room for a lower gap between the dimension of the tangent space and its expected one. On the other hand it is possible a priori to find more singular counterexamples of possibly smaller length on $\Hilb^d\A^3$.}

We also remark that the construction in \cite{GGGL} is valid in any characteristic other than two, where the problem is still open. Similarly, \Cref{thmintro:main} is true for an algebraically closed field of any characteristic other than two.

We conclude this section mentioning that, building upon the example given in \cite{GGGL}, a counterexample to another longstanding conjecture had been produced in \cite{behrendnonconstant}. This conjecture was about the constancy of the Behrend function, which is a constructible function attached to any scheme of finite type over the complex numbers \cite{Beh,SignAnn}. This function was explicitly computed in very few cases \cite{Beh,GRICOLFI,SignAnn} and no general method for computing its values is currently known. In the last part of \Cref{sec:parity} we prove that the counterexample given in \cite{behrendnonconstant} is smoothable (\Cref{prop:smoothableBeh}) and we give a counterexample of smaller length and higher rank in \Cref{ex:Beh}.

\subsection*{Nested Hilbert scheme}

Generalisations of Hilbert schemes, are provided by nested Hilbert schemes and, more generally by nested $\Quot$-schemes \cite{NESTQUOT,NESTQUOTFR} or by double nested Hilbert schemes \cite{Mon_double_nested}. In this section, we deal with nested Hilbert schemes of points. Roughly speaking, they are defined as the fine moduli spaces parametrising nestings of zero-dimensional closed subschemes of a given quasi-projective variety.  
While dealing with these objects one suddenly realises that their geometry gets  involved quickly. For instance, nested Hilbert schemes are already reducible when $\dim X\geqslant2$, while double nested Hilbert schemes of points on smooth curves are in general reducible and have many smoothable components, i.e.~components whose generic point corresponds to a nesting involving only smooth subschemes, \cite{dbnhs,ALESSIONESTED}.
Another typical question about Hilbert schemes concerns their schematic structure. Precisely, in many instances it is not known whether they are reduced or, even worse, if they have generically non-reduced components.
Among the wild zoo of irreducible components of the Hilbert schemes of points, a special r\^ole is played by elementary components, i.e. irreducible components parametrising fat points. These components are considered as the building blocks of Hilbert schemes as all components are generically étale-locally products of elementary ones \cite{Iarrocomponent}.
Thanks to \cite{ELEMENTARY}, we know that being generically reduced for an elementary component is equivalent to have the TNT\footnote[2]{The acronym "TNT" stands for \textit{trivial negative tangent}.} property at its general point, i.e.~that the negative tangents at its general point consist of translations only, see \Cref{def:negativetangents}. As for the irreducibility problem, the situation becomes more complicated passing from classical Hilbert schemes to nested ones. Indeed, it is known that the Hilbert scheme of 21 points on a  smooth fourfold admits generically non-reduced components \cite{JoachimQuestions} but the question about reducedness of the Hilbert schemes of points on smooth threefold is still open. 
Analogously, the question about reducedness was open in the case $\dim X=2$, cf.~\cite{JoachimQuestions,ALESSIONESTED}.  We address this problem successfully. In this direction we prove first that, when $\dim X\geqslant4 $ the nested Hilbert schemes become non-reduced as soon as possible. Then, we deduce that nested Hilbert schemes on smooth surfaces may have generically non-reduced components. It  is worth mentioning that we still do not know the minimal length of a nesting producing generically non-reduced components.

We list now our  main results on the topic. We adapt first the theory developed in \cite{ELEMENTARY} to the nested setting, then we deal with $\dim X\geqslant 4$ and finally we consider the surface case.


\begin{introtheorem}[\Cref{thm:tnt per nested}]\label{thmintro:tnt per nested}
Let $\mathbf{d}\in \Z^r$ be any non-decreasing sequence of non-negative integers and let $V\subset \Hilb^{\mathbf{d}} X$ be an irreducible component.  Suppose that $V$ is generically reduced. Then $V$ is elementary if and only if a general point of $V$ has trivial negative tangents.  
\end{introtheorem}


\begin{introtheorem}[\Cref{thm:main2}]\label{thmintro:main2}
    Let $X$ be a smooth quasi-projective variety and let $d\geqslant 2$ be a positive integer. Let $V\subset \Hilb^{d} X$ be a generically reduced elementary component. Then, the nested Hilbert scheme $\Hilb^{(1,d)} X$ has a generically non-reduced elementary component $\widetilde{V}$ such that
    \[
    \widetilde{V}_{\redu} \cong V_{\redu}.
    \]
\end{introtheorem}

\begin{introtheorem}[\Cref{thm:solvereduced}] \label{thmintro:solvereduced}
    Let $S $ be a smooth quasi-projective surface. Then, there exists a non-decreasing sequence of non-negative integers $\mathbf d\in\mathbb Z^r$ such that $S^{[\mathbf{d}]}$ has a generically non-reduced component. 
\end{introtheorem}

An interesting feature of \Cref{thmintro:main2} is that  this phenomenon  arises in many other contexts. For instance, in \cite{Sha90} some elementary components of the classical Hilbert scheme of points are compared with irreducible components of the moduli spaces parametrising the corresponding nilpotent algebras. In this setting, it is proven that the two objects are birational to each other as varieties but not as schemes. Another instance of the phenomenon in \Cref{thmintro:main2} occurs in the construction of the \textit{isospectral Hilbert scheme} where taking reduced structure is crucial in \cite{Hai}.

Finally, we also present an infinite family of elementary components of Hilbert schemes of points, including already known examples (cf.~\cite{Iarrob,ELEMENTARY,Satrianostaal}), as a byproduct of the following theorem.

\begin{introtheorem}[\Cref{prop:newcomp}]\label{thmintro:newcomp}
    Let $R=\C[x_1,\ldots,x_n,y_1,\ldots,y_n]$, for $n\geqslant 2$, be the polynomial ring in $2n$ variables and complex coefficients.
    Then, the ideal
    \[
    I=\sum_{i=1}^n (x_i,y_i)^2+ (x_1\cdots x_n -y_1\cdots y_n),\]
    has TNT. Therefore, by \cite[Theorem 4.5]{ELEMENTARY} every component of the Hilbert scheme $\Hilb^{3^n-1} \A^{2n}$ containing it is elementary.
\end{introtheorem}

\subsection*{Organisation of the content}  \Cref{sec:parity} concerns the parity conjecture for $\Quot$-schemes. In \Cref{subsec:originalexample} we explain our strategy to \lq\lq explore\rq\rq\  the $\Quot$-scheme by deforming monomial submodules. Then, in \Cref{subsec:counterexample} we prove the main result of this paper about $\Quot$-schemes, \Cref{thmintro:main} (\Cref{thm:main}) and in \Cref{subsec:explanation} we explain an alternative strategy based on the manipulation of the ideals given in \cite{GGGL}. In \Cref{subsec:bfunction} we prove in \Cref{prop:smoothableBeh} smoothability of the counterexample {to the constancy of the Behrend function} given in \cite{behrendnonconstant} and we give a new example for $\Quot_3^{13}\A^3$. In \Cref{sec:nested} we deal with nested Hilbert schemes. Precisely, in \Cref{subsec:TNT} we recall the theory of negative tangents by presenting it in the nested setting and we prove \Cref{thmintro:tnt per nested} (\Cref{thm:tnt per nested}). Then, in \Cref{subsec:nonredcomp} we  prove \Cref{thmintro:main2} (\Cref{thm:main2}), \Cref{thmintro:solvereduced} (\Cref{thm:solvereduced}), and we discuss one example. Finally, in \Cref{subsec:newelem} we give an infinite family of elementary components of the Hilbert schemes of points by proving \Cref{thmintro:newcomp} (\Cref{prop:newcomp}). 
 
\subsection*{Acknowledgments}
 
	We thank Joachim Jelisiejew for the encouragement to write this paper, and the help provided for \Cref{sec:nested}. We thank Rahul Pandharipande for the interesting conversation about the parity conjecture.
 We thank Andrea Ricolfi for his precious help and for the countless conversations on the Behrend function. We thank Matthew Satriano for very fruitful discussions. We thank Alessio Sammartano for the support he gave to the project from the very first moments.

	\section{New counterexamples to the parity conjecture}\label{sec:parity}
After having established the most convenient notation for us, in the first part of this section we explain a possible strategy to produce counterexamples to the party conjecture. Then, we prove \Cref{thmintro:main} of the introduction and we explain a possible different approach which also produces examples of $\Quot$-schemes on smooth threefold with non-constant Behrend function. Finally, we show the smoothability of the ideal given in \cite{behrendnonconstant}.
 
	Let $r,d\in\Z$ be two positive integers, and let $X$ be a smooth quasi-projective variety. Recall that the $(r,d)$-\textit{$\Quot$-functor} of $X$ denoted by $\underline{\Quot}_r^d X:\Sch^{\opp}_{\C} \to \Sets$, is the contravariant functor defined as follows
	\[
	\left(\underline{\Quot}_r^d X\right)(S) = \Set{ \pi_X^*\Calo_X^{\oplus r}\twoheadrightarrow \mathcal G |\mathcal G  \in  \Coh (X\times S)     \mbox{ is $S$-flat, $S$-finite, } \len_S\mathcal{G} =d }/\sim,
	\] 
	where $\pi_X:X\times S \rightarrow X$ denotes the canonical projection,  $\len_S\mathcal{G}$ denotes the relative length of $\mathcal{G}$ over $S$, and two surjections are isomorphic if they have the same kernel. By a celebrated result of Grothendieck, the functor $\underline{\Quot}_r^d X$ is representable  and the fine moduli space $\Quot_r^dX$ representing it is a quasi-projective scheme called \textit{Quot-scheme} \cite{FGA}. Whenever not needed, we will omit the surjection and we will denote points of the $\Quot$-scheme by $[\Calf]$. In this setting, if
 \[
 [0\rightarrow \Calk\rightarrow\Calo_X^{\oplus r}\rightarrow \Calf\rightarrow 0]\in\Quot_r^d X 
 \]
 is any point, there is a natural identification \cite[Theorem 6.4.9]{FGAexplained}
 \begin{equation}
 \label{eq:tangentquot}
     \mathsf T_{[\Calf]} \Quot_r^d X\cong \Hom_{\Calo_X}(\Calk , \Calf).
 \end{equation}
 \begin{remark}\label{rem:etale}
     As already anticipated in the introduction, since our questions are local in nature, it is fair to put $X\cong \A^3_\C$ and hence to work up to étale covers. We will denote by $R=\C[x,y,z]$ the polynomial ring in three variables and complex coefficients. Recall that given any $R$-module $M$, a submodule $N\subset M$, and an ideal $J\subset R$, the colon module is 
     \[
     (N:_MJ)=\Set{ m\in M| J\cdot m\subset N}.
     \]
     We omit the subscript $M$ from the notation above whenever no confusion is possible.
     If in addition $M$ is finite dimensional over $\mathbb C$, the colength of $N$ is the integer $\dim_\C (M/N)= \len (M/N)$, see \cite{EISENBUD} for more details.
 \end{remark}
 
 The following result was proven in \cite{SAMMARTANO}. It gives a positive answer to the Parity conjecture (\Cref{conjecture2}) in the homogeneous setting.
 \begin{theorem}[{\cite[Theorem 10]{SAMMARTANO}}]\label{thm:partyhomo}
     Let $\mathbf{k}$ denote an algebraically closed field of any characteristic and let $R = \mathbf{k}[x,y,z]$ be the polynomial ring in three variables. Let also $F$ be a free graded $R$-module of rank $r$ and let $M = F/K$ be a
graded $R$-module of finite length $d$. Then,
\[
\dim_\mathbf{k}\Hom_R(K, M) \equiv rd \pmod 2.
\] 
 \end{theorem} 
As proven by the authors in \cite{GGGL}, \Cref{thm:partyhomo} does not generalise to the non-homogeneous setting for $d\geqslant12$. It is remarkable that the counterexamples given in \cite{GGGL} are valid in characteristic different from two, where the parity conjecture is still open.

\subsection{The behind-the-scenes at the original counterexample for Hilbert schemes}\label{subsec:originalexample}
We start with a brief analysis of the family of counterexamples described in \cite{GGGL}. The general member of this family has the following form
\begin{equation}
    \label{eq:generalcount}
   \left(x, (y,z)^2\right)^2 + \left(y^3 + b_1y^2z + b_2yz^2 + b_3z^3 + b_4 xy+ b_5xz\right) \subset R,
\end{equation} 
for a generic choice of complex parameters $b_i\in\C$, for $i=1,\ldots,5$. These ideals are obtained by deforming the monomial ideal
\begin{equation} \label{eq:monomialIdeal}
J = \left(x, (y,z)^2\right)^2 + \left(y^3\right)  = \left(x^2,xy^2, y^3, xyz,xz^2,y^2z^2,yz^3,z^4\right),
\end{equation}
by adding to the generator $y^3$ an element of $(J:(x,y,z))$, i.e.~a polynomial not contained in $J$ which enters $J$ if multiplied by any variable. More precisely, the family of ideals in \eqref{eq:generalcount} shows that there exists a vector subspace $V \subset \mathsf{T}_{[J]}\Hilb^{12} \mathbb{A}^3$ of dimension at least 5 of unobstructed first order deformations.

This observation suggests to look for new counterexamples to the parity conjecture as we explain here. As the $\Quot$-scheme can be covered with charts, each containing precisely one monomial submodule, we adopt the following strategy:
\begin{description}
    \item[Step 1] compute all the monomial submodules in $R^{\oplus r}$ of given colength $d$;
    \item[Step 2] for every monomial submodule try to determine a subspace of the tangent space of unobstructed first order deformations;
    \item[Step 3] for every subspace of unobstructed first order deformations consider a random vector and check whether the generic fiber of the corresponding $\A^1$-deformation gives a new counterexample.
\end{description}

Step 1 provides one point for each chart and allows to analyse all of them. In order to speed up the search, one can take into account the action of the projective linear group by considering only the Borel-fixed submodules \cite{ABRS}.\footnote[3]{Borel-fixed submodules can be efficiently computed via Borel-fixed ideals with the \textit{Macaulay2} \cite{M2} package \texttt{StronglyStableIdeals} described in \cite{AL}.}

Step 2 and Step 3 allow to \lq\lq get far\rq\rq~enough from monomial submodules  to obtain a  submodule  not homogeneous with respect to any grading, so that it does not satisfy the necessary conditions for the validity of the parity conjecture given in \cite{SAMMARTANO}.

\medskip   

Combining Step 1, 2 and 3 we obtain a computationally efficient procedure to select random points from the parameter space. This approach worked well for finding new counterexamples to the parity conjecture but in principle it can be used to investigate other properties.

\begin{example}
Applying this strategy, one can determine a lot of other counterexamples to the parity conjecture for the Hilbert scheme $\Hilb^d \mathbb{A}^3$. In \Cref{fig:tgtspaces Hilb}, we list pairs of values $(d,t)$ corresponding to new counterexamples $[Z] \subset \Hilb^d \mathbb{A}^3$ with $t = \dim_\C  \mathsf{T}_{[Z]} \Hilb^d \mathbb{A}^3 \not\equiv d \pmod{2}$ not coming from a counterexample in $\Hilb^{d-1} \mathbb{A}^3$.\footnote[4]{An explicit example of each case is available in the ancillary \textit{Macaulay2} \cite{M2} file \href{www.paololella.it/software/list-of-counterexamples-Hilbert-scheme.m2}{\tt list-of-counterexamples-Hilbert-scheme.m2}} Indeed, recall that if $[Z] \in \Hilb^d \mathbb{A}^3$ is a counterexample, then for every point $p \in \mathbb{A}^3$ not contained in the support of $Z$, the point $[Z \cup p] \in \Hilb^{d+1} \mathbb{A}^3$ is a counterexample, as $\dim_\C  \mathsf{T}_{[Z \cup p]} \Hilb^{d+1} \mathbb{A}^3 = \dim_\C  \mathsf{T}_{[Z]} \Hilb^{d} \mathbb{A}^3  + \dim_\C  \mathsf{T}_{[ p]} \Hilb^{1} \mathbb{A}^3 = \dim_\C  \mathsf{T}_{[Z]} \Hilb^{d} \mathbb{A}^3  + 3$.
\end{example}

\begin{table}[!ht]
\begin{tikzpicture}[xscale=0.925,yscale=0.9]
\draw (-1,0) -- (14,0);
\node at (-0.5,0.5) [] {$d$};  
\node at (-0.5,-0.6) [] {$t$};  

\node at (0.5,0.5) [] {$\leqslant 11$};
\node at (1.5,0.5) [] {$12$};
\node at (2.5,0.5) [] {$13$};
\node at (3.5,0.5) [] {$14$};
\node at (4.5,0.5) [] {$15$};
\node at (5.5,0.5) [] {$16$};
\node at (6.5,0.5) [] {$17$};
\node at (7.5,0.5) [] {$18$};
\node at (8.5,0.5) [] {$19$};
\node at (9.5,0.5) [] {$20$};
\node at (10.5,0.5) [] {$21$};
\node at (11.5,0.5) [] {$22$};
\node at (12.5,0.5) [] {$23$};
\node at (13.5,0.5) [] {$24$};

\draw (0,1) -- (0,-1.25);
\draw (1,1) -- (1,-1.25);
\draw (2,1) -- (2,-1.25);
\draw (3,1) -- (3,-1.25);
\draw (4,1) -- (4,-1.25);
\draw (5,1) -- (5,-1.25);
\draw (6,1) -- (6,-1.25);
\draw (7,1) -- (7,-1.25);
\draw (8,1) -- (8,-1.25);
\draw (9,1) -- (9,-1.25);
\draw (10,1) -- (10,-1.25);
\draw (11,1) -- (11,-1.25);
\draw (12,1) -- (12,-1.25);
\draw (13,1) -- (13,-1.25);

    \begin{scope}[shift={(0,-0.5)},black!70]
    \draw[-latex] (0.8,0.125) to[bend left=40]node[fill=white,inner sep=0.5pt,yshift=4pt]{\tiny ${+}3$} (1.2,0.125);
    \draw[-latex] (1.8,0.125) to[bend left=40]node[fill=white,inner sep=0.5pt,yshift=4pt]{\tiny ${+}3$} (2.2,0.125);
    \draw[-latex] (2.8,0.125) to[bend left=40]node[fill=white,inner sep=0.5pt,yshift=4pt]{\tiny ${+}3$} (3.2,0.125);
    \draw[-latex] (3.8,0.125) to[bend left=40]node[fill=white,inner sep=0.5pt,yshift=4pt]{\tiny ${+}3$} (4.2,0.125);
    \draw[-latex] (4.8,0.125) to[bend left=40]node[fill=white,inner sep=0.5pt,yshift=4pt]{\tiny ${+}3$} (5.2,0.125);
    \draw[-latex] (5.8,0.125) to[bend left=40]node[fill=white,inner sep=0.5pt,yshift=4pt]{\tiny ${+}3$} (6.2,0.125);
    \draw[-latex] (6.8,0.125) to[bend left=40]node[fill=white,inner sep=0.5pt,yshift=4pt]{\tiny ${+}3$} (7.2,0.125);
    \draw[-latex] (7.8,0.125) to[bend left=40]node[fill=white,inner sep=0.5pt,yshift=4pt]{\tiny ${+}3$} (8.2,0.125);
    \draw[-latex] (8.8,0.125) to[bend left=40]node[fill=white,inner sep=0.5pt,yshift=4pt]{\tiny ${+}3$} (9.2,0.125);
    \draw[-latex] (9.8,0.125) to[bend left=40]node[fill=white,inner sep=0.5pt,yshift=4pt]{\tiny ${+}3$} (10.2,0.125);
    \draw[-latex] (10.8,0.125) to[bend left=40]node[fill=white,inner sep=0.5pt,yshift=4pt]{\tiny ${+}3$} (11.2,0.125);
    \draw[-latex] (11.8,0.125) to[bend left=40]node[fill=white,inner sep=0.5pt,yshift=4pt]{\tiny ${+}3$} (12.2,0.125);
    \draw[-latex] (12.8,0.125) to[bend left=40]node[fill=white,inner sep=0.5pt,yshift=4pt]{\tiny ${+}3$} (13.2,0.125);
    \end{scope}

  \node at (0.5,-0.65) [black!50] {\small none};
  
  \node at (1.5,-0.65) [] {\small 45};

  \node at (4.5,-0.65) [] {\small 60};

  \node at (5.5,-0.6) [] {\small 55};
  \node at (5.5,-0.9) [] {\small 61};

  \node at (6.5,-0.65) [] {\small 74};

  \node at (7.5,-0.65) [] {\small 75};
  
  \node at (10.5,-0.65) [] {\small 90};

  \node at (11.5,-0.65) [] {\small 91};

  \node at (12.5,-0.65) [] {\small 98};

\node at (13.5,-0.6) [] {\small 103};
\node at (13.5,-0.9) [] {\small 121};

\end{tikzpicture}
\caption{Counterexamples found with the strategy in \Cref{subsec:originalexample} for $\Hilb^d \mathbb{A}^3$.}\label{fig:tgtspaces Hilb}
\end{table}

\subsection{New counterexamples for the Quot scheme}\label{subsec:counterexample}

We now introduce some notation and then state a criterion to produce many unobstructed first order deformations based only on the combinatorics of monomial ideals. These preliminary results hold for a polynomial ring with any number of variables. Thus, we start denoting by $R$ be the polynomial ring $\C[x_1,\ldots,x_n]$ in $n$ variables and complex coefficients and we will restrict to the case $n=3,\ R = \C[x,y,z]$ when dealing with the parity conjecture.

Let $J \subset S$ be a monomial ideal with finite colength. We denote by $\mathcal{B}_J\subset R$ the set of monomial generators of $J$ and by $\mathcal{N}_J\subset R$ the unique monomial basis of $R/J$, i.e.~the unique set of monomials whose image in $R/J$ consists of a basis. We will often abuse of notation and denote with the same symbol elements in $\mathcal{N}_J$ and their image in $R/J$.

A \emph{monomial submodule} $U\subset R^{\oplus r}$ is a submodule generated by \textit{monomial terms}, i.e.~elements of the form $m \mathbf{e}_i\in R^{\oplus r}$ where $m$ is a monomial in $R$ and $\mathbf{e}_i$ is any element of the canonical basis of $R^{\oplus r}$. Hence, every monomial submodule $U$ can be written as direct sum $\bigoplus_{i=1}^r J_i\mathbf{e}_i \subseteq \bigoplus_{i=1}^r R\mathbf{e}_i = R^{\oplus r}$ where $J_i$'s are monomial ideals in $R$.  We will only consider monomial submodules $U = \bigoplus_{i=1}^r J_i \mathbf{e}_i$ of finite colength, so that $\text{lenght}(R^{\oplus r}/U) = \sum_{i=1}^r \text{length}(R/J_i)$. 

As monomial ideals, a monomial submodule $U\subset R^{\oplus r}$ admits a unique minimal set of monomial generators, which we denote by $\mathcal B_U\subset R^{\oplus r}$. Similarly, we denote by $\mathcal N_U\subset R^{\oplus r}$  the unique monomial basis of $R^{\oplus r}/U$, i.e.~the unique minimal set of monomial terms whose image in $R^{\oplus r}/U$ constitutes a $\C$-basis. 

If we denote by $d$ the colength of the submodule $U$, then we can interpret it as a point
\[
[ 0 \to U \to R^{\oplus r} \to R^{\oplus r}/U \to 0] \in \Quot^d_r\mathbb{A}^n.
\]
As we have fixed the basis $\mathcal{N}_U$ for the quotient $R^{\oplus r}/U$, we can explicitly describe a tangent vector $\varphi \in \mathsf{T}_{[U]} \Quot^d_r\mathbb{A}^n = \Hom_R(U,R^{\oplus r}/U)$ by associating to every monomial generator of $U$ a linear combination of the elements in $\mathcal{N}_U$, i.e.~for every $b\in \mathcal{B}_U$
\begin{equation}\label{eq:maptg}
    \varphi(b) = \sum_{s \in \mathcal{N}_U} \gamma_{b,s} s,\qquad \gamma_{b,s} \in \mathbb{C}.
\end{equation}

We stress that the choice of the coefficients $\gamma_{b,s}$ is not arbitrary and it is subject to the conditions coming from the syzygies of $U$. Let  $R_\epsilon = \mathbb{C}[\epsilon][x_1,\ldots,x_n]$ and $\mathbb{C}[\epsilon] = \mathbb{C}[t]/(t^2)$ denotes the ring of dual numbers ($\epsilon$ is the equivalence class of $t$ in $\mathbb{C}[t]/(t^2)$). Given a tangent vector $\varphi\in\mathsf{T}_{[U]} \Quot^d_r\mathbb{A}^n$ presented as in \eqref{eq:maptg}, one can define a $\C[\epsilon]$-family of $R$-modules via the following submodule
\[
\left( b - \epsilon\cdot \varphi(b) \ \middle\vert\  b \in \mathcal{B}_U\right) \subset R_\epsilon^{\oplus r}.
\]

In what follows we shall need the following definition.

\begin{definition} 
Given a monomial submodule $U \subset R^{\oplus r}$ of finite colength, the \emph{socle} of $U$ is the set of monomial terms
\[
\mathcal{S}_U  = \mathcal{N}_U\cap \left(U:_{R^{\oplus r}}(x_1,\ldots,x_n)\right).
\]
\end{definition} 

Considering the decomposition $U = \bigoplus_{i=1}^r J_i\mathbf{e}_i$, we have that $\mathcal{S}_U = \bigcup_{i=1}^r \mathcal{S}_{J_i}\mathbf{e}_i$ where $\mathcal{S}_{J_i} = \mathcal{N}_{J_i} \cap \big(J_i :_R (x_1,\ldots,x_n)\big)$.

 \begin{remark} \label{rem:socleisflat}
  For any choice of $\gamma_{b,s}\in\C$, for $b\in\mathcal{B}_U$ and $s\in\mathcal{S}_U$, the  submodule
\[
\left( b - \epsilon\sum_{s \in \mathcal{S}_U} \gamma_{ b,s}\cdot s \ \middle\vert\  b \in \mathcal{B}_U\right) \subset R_\epsilon^{\oplus r}
\]
defines a first order deformation, i.e.~a $\C[\epsilon]$-family. This is true because the syzygies of $I$ do not impose conditions on the socle elements.
\end{remark}
 
 Recall that a first order deformation $\varphi$ is \emph{unobstructed}\footnote[5]{Notice that this is not the classical notion of unobstructed tangent vector. However, we adopt this terminology for the sake of exposition.} if the submodule generated by
\[
\left( b - t \sum_{s \in \mathcal{N}_U} \gamma_{b,s} \cdot s\ \middle\vert\  b \in \mathcal{B}_U\right) \subset R_t^{\oplus r}
\]
is flat over $\mathbb{A}^1 = \Spec \mathbb{C}[t]$, here $R_t$ denotes the polynomial ring $\C[t][x_1,\ldots,x_n]$.
 
\begin{lemma}\label{lem:unobstructed}
Let $U   \subseteq R^{\oplus r}$ be a monomial submodule of finite colength. Consider two subsets $B \subseteq \mathcal{B}_U$ and $S \subseteq \mathcal{S}_U$ such that 
\[
x_\ell \cdot s \notin B,\quad\mbox{ for all } s\in S,\mbox{ and }\ell=1,\ldots,n.
\]
Then, every function $\phi: \mathcal{B}_U \to \text{Span}_\C(\mathcal{N}_U)$ of the following form
\[
\phi(b) = \begin{cases}
    0 & \text{if~}b \notin B,\\
    \displaystyle\sum_{s \in S} \gamma_{b,s} \cdot s & \text{if~}b \in B,
\end{cases}
\]
induces an unobstructed first order deformation in $\mathsf{T}_{[U]} \Quot^d_r(\mathbb{A}^n)$.
\end{lemma}
\begin{proof}
The map $\phi$ defines a first order deformation as per \Cref{rem:socleisflat}.
In order to prove the statement, we show that the submodule $M \subset R^{\oplus r}_t$ generated by
    \[
     (\mathcal{B}_U\setminus B) \cup \Set{b - t \phi(b)|b \in B}
    \]
    defines a flat family over the affine line $\mathbb{A}^1_\C = \Spec \mathbb{C}[t]$. This is the case if every relation among the generators of $U$ lifts to a relation among the generators of $M$ \cite[Section 1.3]{Artin}. Let 
    \begin{equation}\label{eq:monomial syzygy}
    \sum_{b \in \mathcal{B}_U} P_{b} \cdot b = \sum_{b \in \mathcal{B}_U\setminus B}    P_{b} \cdot b + \sum_{b \in B}   P_{b} \cdot b = 0,\qquad P_{b} \in R    
    \end{equation}
    be a syzygy among the generators of $U$. By definition of socle,   $S$ and $B$, if $b_0\in   B$ then $P_{b_0}\cdot \phi(b_0)$ belongs to the monomial submodule $M'\subset U$ generated by the terms in $\mathcal{B}_U\setminus B$, that is 
    \[
    P_{b_0}\cdot \phi(b_0) = \sum_{b \in \mathcal{B}_U\setminus B} Q_{b_0,b}\cdot b ,\qquad Q_{b_0,b} \in R,\ b\in B.
    \]
    As a consequence, we have
    \[
    \sum_{c \in B} P_{c}\cdot \phi(c) = \sum_{c \in B} \left( \sum_{b \in \mathcal{B}_U\setminus B} Q_{c,b}\cdot b\right) = \sum_{b \in \mathcal{B}_U\setminus B} \left( \sum_{c \in B} Q_{c,b} \right)\cdot b.
    \]
    Then, the syzygy \eqref{eq:monomial syzygy} lifts to
    \[
    \sum_{b \in \mathcal{B}_U\setminus B}   \widetilde{P}_{b} \cdot b + \sum_{b \in B}  P_{b}\cdot  \big(b - t \phi(b)\big) = 0 
    \]
    where, for $b\in \mathcal{B}_U\setminus B$, we put
    \[
    \widetilde{P}_{b} = P_{b} + t\left( \sum_{c \in B} Q_{c,b} \right). \qedhere
    \]
\end{proof}

We prove now the main theorem of this section, \Cref{thmintro:main} of the introduction. The explicit counterexample given in the proof is obtained following the strategy just described. In a similar way we found other examples that we briefly mention in \Cref{fig:tgtspaces}.

\begin{theorem}\label{thm:main}
    The parity conjecture (\Cref{conjecture2}) is false when $(r,d)$ runs in the following range:
    \begin{itemize}
        \item $r=1$ and $d\geqslant12$,
        \item $r\geqslant 2$ and $d\geqslant8$.
    \end{itemize}
\end{theorem}
\begin{proof}
The case $r=1$ is the content of {\cite[Corollary 3.2]{GGGL}}. For $r \geqslant 2$, the first counterexample has length 8.

First, we describe how to use the strategy depicted in the previous subsection to obtain the counterexample. Second, we show how to produce a counterexample for $(r+1,d)$ and $(r,d+1)$ from a counterexample for $(r,d)$.

Consider the monomial submodule $U = J_1\mathbf{e}_1 \oplus J_2\mathbf{e}_2 \subseteq R^{\oplus 2}$ with
\[
J_1 = (x,y^2,yz^2,z^3) \qquad\text{and}\qquad J_2 = (x,y^2,yz,z^2).
\]
We have
\[
\mathcal{B}_U = \left\{x\,\mathbf{e}_1,y^2\,\mathbf{e}_1,yz^2\,\mathbf{e}_1, z^3\,\mathbf{e}_1\right\} \cup \left\{x\,\mathbf{e}_2,y^2\,\mathbf{e}_2,yz\,\mathbf{e}_2,z^2\,\mathbf{e}_2\right\}\quad\text{and}\quad \mathcal{S}_U = \left\{yz\, \mathbf{e}_1, z^2\, \mathbf{e}_1 \right\} \cup \left\{y\, \mathbf{e}_2,z\,\mathbf{e}_2\right\}.
\]
The pair $B = \{x\,\mathbf{e}_2,y^2\,\mathbf{e}_2,z^2\,\mathbf{e}_2\}$ and $S = \{yz\, \mathbf{e}_1, z^2\, \mathbf{e}_1\}$ satisfies the hypothesis of \Cref{lem:unobstructed}. Consider the function $\phi:\mathcal{B}_U \to \text{Span}_\C(\mathcal{N}_U)$ 
\[
\phi(x\,\mathbf{e}_1) = \phi(y^2\,\mathbf{e}_1) = \phi(yz^2\,\mathbf{e}_1) = \phi( z^3\,\mathbf{e}_1) = \phi(yz\,\mathbf{e}_2) = 0,\quad \phi(x\mathbf{e}_2) = \phi(z^2\mathbf{e}_2) = yz+z^2,\quad \phi(y^2\mathbf{e}_2) = yz-z^2
\]
and let $M_t \subset R_t^{\oplus r}$ be the associated deformation of $U$ generated by
\[
\left\{ x\,\mathbf{e}_1, y^2\,\mathbf{e}_1,yz^2\,\mathbf{e}_1, z^3\,\mathbf{e}_1\right\} \cup \left\{x\,\mathbf{e}_2 + t(yz+z^2)\mathbf{e}_1,y^2\,\mathbf{e}_2 +  t(yz-z^2)\mathbf{e}_1,yz\,\mathbf{e}_2,z^2\,\mathbf{e}_2 + t(yz+z^2)\mathbf{e}_1\right\}.
\]
The generic fiber $M$ defines a point $[ 0 \to M \to R^{\oplus 2} \to R^{\oplus 2}/M \to 0] \in \Quot^8_2 \A^3$ with tangent space of dimension $37 \not \equiv 2\cdot8 \pmod 2$.\footnote[6]{See the ancillary \textit{Macaulay2} \cite{M2} file \href{www.paololella.it/software/first-counterexample-Quot-scheme.m2}{\tt first-counterexample-Quot-scheme.m2} for the computation of $\dim_\C \mathsf{T}_{[M]} \Quot^8_2\A^3$.}
 
\smallskip

\underline{Increase the rank.} Let $M \subset R^{\oplus r}$ be a submodule such that $[0 \to M \to R^{\oplus r} \to R^{\oplus r}/M \to 0] \in \Quot^d_r \A^3$ is a counterexample to the parity conjecture. The submodule $M \oplus R \subset R^{\oplus r+1}$ defines a point 
\[
[0 \to M\oplus R \to R^{\oplus r+1} \to R^{\oplus r+1}/(M\oplus R) \simeq R^{\oplus r}/M  \to 0] \in \Quot^d_{r+1} \A^3
\]
with tangent space
\[
\begin{split}
&\Hom_R\big(M\oplus R, R^{\oplus r+1}/(M\oplus R)\big)\simeq \Hom_R(M\oplus R, R^{\oplus r}/M) \\
&\qquad\simeq \Hom_R(M,R^{\oplus r}/M)\oplus \Hom_R(R ,R^{\oplus r}/M)\simeq \Hom_R(M,R^{\oplus r}/M)\oplus (R^{\oplus r}/M)
\end{split}
\]
of dimension
\[
\dim_\C  \mathsf{T}_{[M\oplus R]} \Quot_{r+1}^{d}\A^3 = \dim_\C  \mathsf{T}_{[M]} \Quot_{r}^{d}\A^3 + \dim_\C  (R^{\oplus r}/M) = \dim_\C  \mathsf{T}_{[M]} \Quot_{r}^{d}\A^3 + d.
\]
Then, $\dim_\C  \mathsf{T}_{[M]} \Quot_{r}^{d}\A^3 \not\equiv rd \pmod 2$ implies $\dim_\C  \mathsf{T}_{[M]} \Quot_{r}^{d}\A^3 + d \not\equiv (r+1)d \pmod 2$.

\smallskip

\underline{Increase the length.} Let $M \subset R^{\oplus r}$ be a submodule such that $[0 \to M \to R^{\oplus r} \to R^{\oplus r}/M \to 0] \in \Quot^d_r \A^3$ is a counterexample to the parity conjecture and assume that the origin does not belong to the support of $R^{\oplus r}/M$. Moreover, consider the submodule $N = (x,y,z) \oplus R^{\oplus (r-1)} \subset R^{\oplus r}$ defining the point $[0 \to N \to R^{\oplus r} \to R^{\oplus r}/N \to 0] \in \Quot^1_r \A^3$. The module $R^{\oplus r}/N \simeq R/(x,y,z)$ is supported at the origin so that
\[
\Ann(R^{\oplus r}/M)+\Ann(R^{\oplus r}/N)=R,
\]
where $\Ann (R^{\oplus r}/M)$ and $\Ann (R^{\oplus r}/N)$ denote the annihilator ideals of $R^{\oplus r}/M$ and $R^{\oplus r}/N$. In particular, one can write $1 = m+n \in R$ for some $m\in \Ann (R^{\oplus r}/M)$ and $n\in\Ann (R^{\oplus r}/N)$.

Let us denote by $\pi_M: R^{\oplus r}\to R^{\oplus r}/M$ and $\pi_N: R^{\oplus r}\to R^{\oplus r}/N$ the projections to the quotient modules. Then, the following composition
\[
\begin{tikzcd}[row sep=tiny, column sep=huge]
    R^{\oplus r}   \arrow[r,"\Delta"] & R^{\oplus r} \oplus R^{\oplus r} \arrow[r,"(\pi_M{,}\pi_N)"] & (R^{\oplus r}/M) \oplus (R^{\oplus r}/N) \\
    \mathbf{e}_i  \arrow[r,mapsto] & (\mathbf{e}_i,\mathbf{e}_i) \arrow[r,mapsto] & \big(\pi_M(\mathbf{e}_i),\pi_N(\mathbf{e}_i)\big)
\end{tikzcd}
\]
 is still a surjection. Indeed
\[
(1-m)\cdot (\pi_M,\pi_N)\circ \Delta(\mathbf{e}_i)=(\pi_M (\mathbf{e}_i),0)\quad\text{and}\quad (1-n)\cdot (\pi_M,\pi_N)\circ \Delta(\mathbf{e}_i)=(0,\pi_N(\mathbf{e}_i)).
\]
Thus, the kernel $K=\ker ((\pi_M,\pi_N)\circ \Delta)=\ker \pi_M\cap\ker\pi_N = M \cap N$ defines a point $[ 0 \to K \to R^{\oplus r} \to R^{\oplus r}/K \simeq (R^{\oplus r}/M) \oplus (R^{\oplus r}/N) \to 0 ]\in\Quot_{r}^{d+1} \A^3$. The tangent space is
\[
\Hom_R(K,(R^{\oplus r}/M) \oplus (R^{\oplus r}/N)) \cong\Hom_R(K,R^{\oplus r}/M)\oplus \Hom_R(K, R^{\oplus r}/N)\cong\Hom_R(M,R^{\oplus r}/M)\oplus \Hom_R(N,R^{\oplus r}/N)
\]
(the second equality can be checked by localising at the points in the support of $R^{\oplus r}/M$ and $R^{\oplus r}/N$) and its dimension is
\[
\dim_{\C}\mathsf T_{[K]} \Quot_r^{d+1}\A^3 = \dim_{\C}\mathsf T_{[M]} \Quot_r^{d}\A^3 + \dim_{\C}\mathsf T_{[N]} \Quot_r^{1}\A^3 = \dim_{\C}\mathsf T_{[M]} \Quot_r^{d}\A^3 + r + 2.
\]
Finally, $\dim_{\C}\mathsf T_{[M]} \Quot_r^{d}\A^3 \not\equiv rd \pmod 2$ implies $\dim_{\C}\mathsf T_{[M]} \Quot_r^{d}\A^3 + r+2 \not\equiv r(d+1) \pmod 2$.
\end{proof}

\Cref{fig:tgtspaces} encodes the dimension of the tangent spaces of other counterexamples to the parity conjecture for rank $1\leqslant r \leqslant 4$ and degree $8\leqslant d\leqslant14$ that are not obtained from counterexamples of rank $r-1$ and degree $d$ or rank $r$ and degree $d-1$ via the procedure described in the proof of \Cref{thm:main}.\footnote[7]{An explicit example of each case is available in the ancillary \textit{Macaulay2} \cite{M2} file \href{www.paololella.it/software/list-of-counterexamples-Quot-scheme.m2}{\tt list-of-counterexamples-Quot-scheme.m2}.}

\begin{table}[H]
\begin{tikzpicture}[xscale=1.75,yscale=-1]
   \draw (-1,0) -- (7,0); 
   \draw (-1,1) -- (7,1); 
   \draw (-1,2) -- (7,2); 
   \draw (-1,3) -- (7,3); 

   \draw (0,-1) -- (0,4);
   \draw (1,-1) -- (1,4);
   \draw (2,-1) -- (2,4);
   \draw (3,-1) -- (3,4);
   \draw (4,-1) -- (4,4);
   \draw (5,-1) -- (5,4);
   \draw (6,-1) -- (6,4);

   \node at (-0.5,0.5) [] {$r=1$}; 
   \node at (-0.5,1.5) [] {$r=2$}; 
   \node at (-0.5,2.5) [] {$r=3$}; 
   \node at (-0.5,3.5) [] {$r=4$}; 

   \node at (0.5,-0.5) [] {$d=8$};
   \node at (1.5,-0.5) [] {$d=9$};
   \node at (2.5,-0.5) [] {$d=10$};
   \node at (3.5,-0.5) [] {$d=11$};
   \node at (4.5,-0.5) [] {$d=12$};
   \node at (5.5,-0.5) [] {$d=13$};
   \node at (6.5,-0.5) [] {$d=14$};

    \begin{scope}[black!80]
   \draw[-latex] (0.125,0.75) to[bend left=20]node[fill=white,inner sep=0.5pt,xshift=4pt]{\tiny ${+}8$} (0.14,1.225);
   \draw[-latex] (0.125,1.75) to[bend left=20]node[fill=white,inner sep=0.5pt,xshift=4pt]{\tiny ${+}8$} (0.14,2.225);
   \draw[-latex] (0.125,2.75) to[bend left=20]node[fill=white,inner sep=0.5pt,xshift=4pt]{\tiny ${+}8$} (0.14,3.225);        
    \end{scope}
    
    \begin{scope}[shift={(1,0)},black!80]
   \draw[-latex] (0.125,0.75) to[bend left=20]node[fill=white,inner sep=0.5pt,xshift=4pt]{\tiny ${+}9$} (0.14,1.225);
   \draw[-latex] (0.125,1.75) to[bend left=20]node[fill=white,inner sep=0.5pt,xshift=4pt]{\tiny ${+}9$} (0.14,2.225);
   \draw[-latex] (0.125,2.75) to[bend left=20]node[fill=white,inner sep=0.5pt,xshift=4pt]{\tiny ${+}9$} (0.14,3.225);        
    \end{scope}

    \begin{scope}[shift={(2,0)},black!80]
   \draw[-latex] (0.125,0.75) to[bend left=20]node[fill=white,inner sep=0.5pt,xshift=5.5pt]{\tiny ${+}10$} (0.14,1.225);
   \draw[-latex] (0.125,1.75) to[bend left=20]node[fill=white,inner sep=0.5pt,xshift=5.5pt]{\tiny ${+}10$} (0.14,2.225);
   \draw[-latex] (0.125,2.75) to[bend left=20]node[fill=white,inner sep=0.5pt,xshift=5.5pt]{\tiny ${+}10$} (0.14,3.225);        
    \end{scope}

    \begin{scope}[shift={(3,0)},black!80]
   \draw[-latex] (0.125,0.75) to[bend left=20]node[fill=white,inner sep=0.5pt,xshift=5.5pt]{\tiny ${+}11$} (0.14,1.225);
   \draw[-latex] (0.125,1.75) to[bend left=20]node[fill=white,inner sep=0.5pt,xshift=5.5pt]{\tiny ${+}11$} (0.14,2.225);
   \draw[-latex] (0.125,2.75) to[bend left=20]node[fill=white,inner sep=0.5pt,xshift=5.5pt]{\tiny ${+}11$} (0.14,3.225);        
    \end{scope}

    \begin{scope}[shift={(4,0)},black!80]
   \draw[-latex] (0.125,0.75) to[bend left=20]node[fill=white,inner sep=0.5pt,xshift=5.5pt]{\tiny ${+}12$} (0.14,1.225);
   \draw[-latex] (0.125,1.75) to[bend left=20]node[fill=white,inner sep=0.5pt,xshift=5.5pt]{\tiny ${+}12$} (0.14,2.225);
   \draw[-latex] (0.125,2.75) to[bend left=20]node[fill=white,inner sep=0.5pt,xshift=5.5pt]{\tiny ${+}12$} (0.14,3.225);        
    \end{scope}

    \begin{scope}[shift={(5,0)},black!80]
   \draw[-latex] (0.125,0.75) to[bend left=20]node[fill=white,inner sep=0.5pt,xshift=5.5pt]{\tiny ${+}13$} (0.14,1.225);
   \draw[-latex] (0.125,1.75) to[bend left=20]node[fill=white,inner sep=0.5pt,xshift=5.5pt]{\tiny ${+}13$} (0.14,2.225);
   \draw[-latex] (0.125,2.75) to[bend left=20]node[fill=white,inner sep=0.5pt,xshift=5.5pt]{\tiny ${+}13$} (0.14,3.225);        
    \end{scope}

    \begin{scope}[shift={(6,0)},black!80]
   \draw[-latex] (0.125,0.75) to[bend left=20]node[fill=white,inner sep=0.5pt,xshift=5.5pt]{\tiny ${+}14$} (0.14,1.225);
   \draw[-latex] (0.125,1.75) to[bend left=20]node[fill=white,inner sep=0.5pt,xshift=5.5pt]{\tiny ${+}14$} (0.14,2.225);
   \draw[-latex] (0.125,2.75) to[bend left=20]node[fill=white,inner sep=0.5pt,xshift=5.5pt]{\tiny ${+}14$} (0.14,3.225);        
    \end{scope}

    \begin{scope}[shift={(0,0.15)},black!80]
    \draw[-latex] (0.85,0.125) to[bend right=40]node[fill=white,inner sep=0.5pt,yshift=4pt]{\tiny ${+}3$} (1.125,0.15);
    \draw[-latex] (1.85,0.125) to[bend right=40]node[fill=white,inner sep=0.5pt,yshift=4pt]{\tiny ${+}3$} (2.125,0.15);
    \draw[-latex] (2.85,0.125) to[bend right=40]node[fill=white,inner sep=0.5pt,yshift=4pt]{\tiny ${+}3$} (3.125,0.15);
    \draw[-latex] (3.85,0.125) to[bend right=40]node[fill=white,inner sep=0.5pt,yshift=4pt]{\tiny ${+}3$} (4.125,0.15);
    \draw[-latex] (4.85,0.125) to[bend right=40]node[fill=white,inner sep=0.5pt,yshift=4pt]{\tiny ${+}3$} (5.125,0.15);
    \draw[-latex] (5.85,0.125) to[bend right=40]node[fill=white,inner sep=0.5pt,yshift=4pt]{\tiny ${+}3$} (6.125,0.15);
    \end{scope}

    \begin{scope}[shift={(0,1.15)},black!80]
    \draw[-latex] (0.85,0.125) to[bend right=40]node[fill=white,inner sep=0.5pt,yshift=4pt]{\tiny ${+}4$} (1.125,0.15);
    \draw[-latex] (1.85,0.125) to[bend right=40]node[fill=white,inner sep=0.5pt,yshift=4pt]{\tiny ${+}4$} (2.125,0.15);
    \draw[-latex] (2.85,0.125) to[bend right=40]node[fill=white,inner sep=0.5pt,yshift=4pt]{\tiny ${+}4$} (3.125,0.15);
    \draw[-latex] (3.85,0.125) to[bend right=40]node[fill=white,inner sep=0.5pt,yshift=4pt]{\tiny ${+}4$} (4.125,0.15);
    \draw[-latex] (4.85,0.125) to[bend right=40]node[fill=white,inner sep=0.5pt,yshift=4pt]{\tiny ${+}4$} (5.125,0.15);
    \draw[-latex] (5.85,0.125) to[bend right=40]node[fill=white,inner sep=0.5pt,yshift=4pt]{\tiny ${+}4$} (6.125,0.15);
    \end{scope}

        \begin{scope}[shift={(0,2.15)},black!80]
    \draw[-latex] (0.85,0.125) to[bend right=40]node[fill=white,inner sep=0.5pt,yshift=4pt]{\tiny ${+}5$} (1.125,0.15);
    \draw[-latex] (1.85,0.125) to[bend right=40]node[fill=white,inner sep=0.5pt,yshift=4pt]{\tiny ${+}5$} (2.125,0.15);
    \draw[-latex] (2.85,0.125) to[bend right=40]node[fill=white,inner sep=0.5pt,yshift=4pt]{\tiny ${+}5$} (3.125,0.15);
    \draw[-latex] (3.85,0.125) to[bend right=40]node[fill=white,inner sep=0.5pt,yshift=4pt]{\tiny ${+}5$} (4.125,0.15);
    \draw[-latex] (4.85,0.125) to[bend right=40]node[fill=white,inner sep=0.5pt,yshift=4pt]{\tiny ${+}5$} (5.125,0.15);
    \draw[-latex] (5.85,0.125) to[bend right=40]node[fill=white,inner sep=0.5pt,yshift=4pt]{\tiny ${+}5$} (6.125,0.15);
    \end{scope}

        \begin{scope}[shift={(0,3.15)},black!80]
    \draw[-latex] (0.85,0.125) to[bend right=40]node[fill=white,inner sep=0.5pt,yshift=4pt]{\tiny ${+}6$} (1.125,0.15);
    \draw[-latex] (1.85,0.125) to[bend right=40]node[fill=white,inner sep=0.5pt,yshift=4pt]{\tiny ${+}6$} (2.125,0.15);
    \draw[-latex] (2.85,0.125) to[bend right=40]node[fill=white,inner sep=0.5pt,yshift=4pt]{\tiny ${+}6$} (3.125,0.15);
    \draw[-latex] (3.85,0.125) to[bend right=40]node[fill=white,inner sep=0.5pt,yshift=4pt]{\tiny ${+}6$} (4.125,0.15);
    \draw[-latex] (4.85,0.125) to[bend right=40]node[fill=white,inner sep=0.5pt,yshift=4pt]{\tiny ${+}6$} (5.125,0.15);
    \draw[-latex] (5.85,0.125) to[bend right=40]node[fill=white,inner sep=0.5pt,yshift=4pt]{\tiny ${+}6$} (6.125,0.15);
    \end{scope}

    \node at (0.5,1.55) [] {\small 37,39};
    \node at (1.5,1.55) [] {\small 47};
\node at (2.5,1.55) [] {\small 55};
\node at (4.5,1.55) [] {\small 61,65};
\node at (5.5,1.55) [] {\small 71,75};
\node at (6.5,1.55) [] {\small 77,81,83,87};

    \node at (0.5,0.55) [black!50] {\small none};
    \node at (1.5,0.55) [black!50] {\small none};
    \node at (2.5,0.55) [black!50] {\small none};
    \node at (3.5,0.55) [black!50] {\small none};
    \node at (4.5,0.55) [] {\small 45};

    \node at (2.5,2.55) [] {\small 71};
\node at (3.5,2.55) [] {\small 74,78};
\node at (4.5,2.55) [] {\small 85,91,93};
\node at (5.5,2.55) [] {\small 92,100};

\node at (6.5,2.55) [] {\small 99,107,109};

\node at (4.5,3.55) [] {\small 111,115};

\node at (5.5,3.55) [] {\small 119,125};

\node at (6.5,3.55) [] {\small 129,133,139};

\end{tikzpicture}
\caption{Other counterexamples found following the strategy in \Cref{subsec:originalexample} for $\Quot^d_r \mathbb{A}^3$.}\label{fig:tgtspaces}
\end{table}

\subsection{An alternative strategy}\label{subsec:explanation}
In this subsection we provide an alternative strategy to construct counterexamples to the parity conjecture starting from the one given in \cite{GGGL}. Also this strategy considers monomial ideals and produces deformations of monomial submodules, but it allows to preserve along the process additional properties  such as homogeneity with respect some non-standard grading, symmetries\ldots 

We illustrate the strategy starting with the deformation 
\begin{equation}\label{eq:counterexample-binomial}
I=\left((x)+(y,z)^2\right)^2+(y^3-xz) \subset R.
\end{equation}
of the monomial ideal $J$ in \eqref{eq:monomialIdeal}. We can depict the Artinian local algebra $R/I$ as a plane partition with two socle boxes identified by the binomial generator.

\begin{figure}[H]
    \centering
    \begin{tikzpicture}[xscale=0.3,yscale=0.36]
        \draw[,black,fill=red] (0,0) -- (0,3) -- (2,2) -- (2,1) -- (3,0.5) -- (3,-0.5) -- (4,-1) -- (4,-2) -- cycle; 
        \draw[,black] (1,-0.5) -- (1,2.5);
        \draw[,black] (2,-1) -- (2,1);
        \draw[,black] (3,-1.5) -- (3,-0.5);
        \draw[,black] (0,1) -- (3,-0.5);
        \draw[,black] (0,2) -- (2,1);

        \draw [,black,fill=green] (0,3) -- (0,4) -- (1,3.5) -- (1,2.5) -- cycle; 
        \draw [,black,fill=yellow] (0,4) -- (1,3.5) -- (2,4) -- (1,4.5) -- cycle; 
        \draw [,black,fill=yellow,xshift=1cm,yshift=-1.5cm] (0,4) -- (1,3.5) -- (2,4) -- (1,4.5) -- cycle; 
        \draw [,black,fill=yellow,xshift=2cm,yshift=-3cm] (0,4) -- (1,3.5) -- (2,4) -- (1,4.5) -- cycle; 
        \draw [,black,fill=yellow,xshift=3cm,yshift=-4.5cm] (0,4) -- (1,3.5) -- (2,4) -- (1,4.5) -- cycle; 
    
        \draw [,black,fill=blue] (1,3.5) -- (2,4) -- (2,3) -- (1,2.5) -- cycle;     
        \draw [,black,fill=blue,xshift=1cm,yshift=-1.5cm] (1,3.5) -- (2,4) -- (2,3) -- (1,2.5) -- cycle;  
        \draw [,black,fill=blue,xshift=2cm,yshift=-3cm] (1,3.5) -- (2,4) -- (2,3) -- (1,2.5) -- cycle;  
        \draw [,black,fill=blue,xshift=3cm,yshift=-4.5cm] (1,3.5) -- (2,4) -- (2,3) -- (1,2.5) -- cycle;  

        \begin{scope}[shift={(-1,-0.5)}]
            \draw [,black,fill=red] (0,0) -- (0,2) -- (1,1.5) -- (1,-0.5) -- cycle;
            \draw[,black] (0,1) -- (1,0.5);
             \draw [,black,fill=green] (1,-0.5) -- (1,0.5) -- (2,0) -- (2,-1) -- cycle;
             \draw [,black,fill=yellow] (0,2) -- (1,2.5) -- (2,2) -- (1,1.5) -- cycle;
             \draw [,black,fill=yellow,xshift=1cm,yshift=-1.5cm] (0,2) -- (1,2.5) -- (2,2) -- (1,1.5) -- cycle;
             \draw [,black,fill=blue] (1,1.5) -- (2,2) -- (2,1) -- (1,0.5) -- cycle; 
             \draw [,black,fill=blue,xshift=1cm,yshift=-1.5cm] (1,1.5) -- (2,2) -- (2,1) -- (1,0.5) -- cycle; 
        \end{scope}

        \begin{scope}[shift={(14,0)}]
        \draw[,black,fill=red] (0,0) -- (0,3) -- (2,2) -- (2,1) -- (3,0.5) -- (3,-0.5) -- (4,-1) -- (4,-2) -- cycle; 
        \draw[,black] (1,-0.5) -- (1,2.5);
        \draw[,black] (2,-1) -- (2,1);
        \draw[,black] (3,-1.5) -- (3,-0.5);
        \draw[,black] (0,1) -- (3,-0.5);
        \draw[,black] (0,2) -- (2,1);

        \draw [,black,fill=green] (0,3) -- (0,4) -- (1,3.5) -- (1,2.5) -- cycle; 
        \draw [,black,fill=yellow] (0,4) -- (1,3.5) -- (2,4) -- (1,4.5) -- cycle; 
        \draw [,black,fill=yellow,xshift=1cm,yshift=-1.5cm] (0,4) -- (1,3.5) -- (2,4) -- (1,4.5) -- cycle; 
        \draw [,black,fill=yellow,xshift=2cm,yshift=-3cm] (0,4) -- (1,3.5) -- (2,4) -- (1,4.5) -- cycle; 
        \draw [,black,fill=yellow,xshift=3cm,yshift=-4.5cm] (0,4) -- (1,3.5) -- (2,4) -- (1,4.5) -- cycle; 
    
        \draw [,black,fill=blue] (1,3.5) -- (2,4) -- (2,3) -- (1,2.5) -- cycle;     
        \draw [,black,fill=blue,xshift=1cm,yshift=-1.5cm] (1,3.5) -- (2,4) -- (2,3) -- (1,2.5) -- cycle;  
        \draw [,black,fill=blue,xshift=2cm,yshift=-3cm] (1,3.5) -- (2,4) -- (2,3) -- (1,2.5) -- cycle;  
        \draw [,black,fill=blue,xshift=3cm,yshift=-4.5cm] (1,3.5) -- (2,4) -- (2,3) -- (1,2.5) -- cycle;  
        \end{scope}

        \node at (8,1.5) [] {$\sim$};
        \begin{scope}[shift={(11.5,-1.25)}]
            \draw [,black,fill=red] (0,0) -- (0,2) -- (1,1.5) -- (1,-0.5) -- cycle;
            \draw[,black] (0,1) -- (1,0.5);
             \draw [,black,fill=green] (1,-0.5) -- (1,0.5) -- (2,0) -- (2,-1) -- cycle;
             \draw [,black,fill=yellow] (0,2) -- (1,2.5) -- (2,2) -- (1,1.5) -- cycle;
             \draw [,black,fill=yellow,xshift=1cm,yshift=-1.5cm] (0,2) -- (1,2.5) -- (2,2) -- (1,1.5) -- cycle;
             \draw [,black,fill=blue] (1,1.5) -- (2,2) -- (2,1) -- (1,0.5) -- cycle; 
             \draw [,black,fill=blue,xshift=1cm,yshift=-1.5cm] (1,1.5) -- (2,2) -- (2,1) -- (1,0.5) -- cycle; 

             \draw [densely dotted,thick] (1,2.5) -- (2.5,3.25);
             \draw [densely dotted,thick,xshift=1cm,yshift=-0.5cm] (1,2.5) -- (2.5,3.25);
             \draw [densely dotted,thick,xshift=1cm,yshift=-1.5cm](1,2.5) -- (2.5,3.25);
             \draw [densely dotted,thick,xshift=2cm,yshift=-2cm](1,2.5) -- (2.5,3.25);
             \draw [densely dotted,thick,xshift=2cm,yshift=-3cm] (1,2.5) -- (2.5,3.25);
        \end{scope}
    \end{tikzpicture}

    \caption{ Pictorial description of the algebra $R/I$. The socle boxes identified by the binomial generator in \eqref{eq:counterexample-binomial} are the green ones.}\label{fig:counterHilb}
\end{figure}
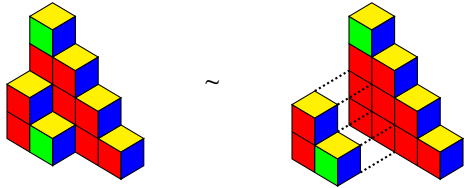

Consider the $R$-module
    \[
    M=\frac{(x,y)+I}{I} 
    \]
corresponds to the point
\[
 [0 \to K \to R^{\oplus 2} \to M \to 0] \in \Quot^8_2 \mathbb{A}^3
\]
where 
\[
K = (x\,\mathbf{e}_1,y^2\,\mathbf{e}_1,yz\,\mathbf{e}_1,z^2\,\mathbf{e}_1,x\,\mathbf{e}_2-y\,\mathbf{e}_1,y^2\,\mathbf{e}_2-z\,\mathbf{e}_1,yz^2\,\mathbf{e}_2,z^3\,\mathbf{e}_2)\subset R^{\oplus 2}.
\]
The dimension of the tangent space turns out to be
\[
\dim_\C \mathsf{T}_{[M]} \Quot^8_2 \mathbb{A}^3 = 39 \not\equiv 2\cdot 8 \pmod 2.
\]
Notice that $K$ is a deformation of the monomial submodule $U = (x,y^2,yz,z^2)\mathbf{e}_1\oplus (x,y^2,yz^2,z^3)\mathbf{e}_2$ and $y\,\mathbf{e}_1,z\,\mathbf{e}_1 \in \mathcal{S}_U$.
 
From a combinatorial viewpoint, the module $M$ is obtained by \textit{\lq\lq removing the four boxes corresponding to the monomials in $\mathcal{N}_J \cap \mathcal{N}_{(x,y)}$ from the partition in \Cref{fig:counterHilb}\rq\rq}. This is explained  in \Cref{fig:square-of-maximal-ideal}.

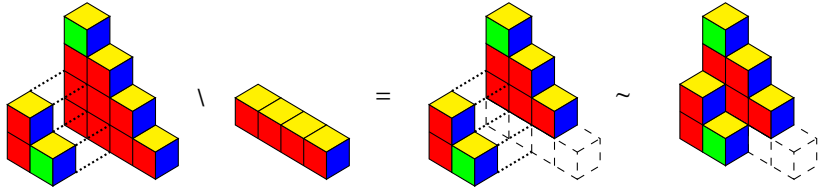
\begin{figure}[!ht]
\centering
\begin{tikzpicture}[xscale=0.3,yscale=0.36]
        \begin{scope}[shift={(0,0)}]
        \draw[,black,fill=red] (0,0) -- (0,3) -- (2,2) -- (2,1) -- (3,0.5) -- (3,-0.5) -- (4,-1) -- (4,-2) -- cycle; 
        \draw[,black] (1,-0.5) -- (1,2.5);
        \draw[,black] (2,-1) -- (2,1);
        \draw[,black] (3,-1.5) -- (3,-0.5);
        \draw[,black] (0,1) -- (3,-0.5);
        \draw[,black] (0,2) -- (2,1);

        \draw [,black,fill=green] (0,3) -- (0,4) -- (1,3.5) -- (1,2.5) -- cycle; 
        \draw [,black,fill=yellow] (0,4) -- (1,3.5) -- (2,4) -- (1,4.5) -- cycle; 
        \draw [,black,fill=yellow,xshift=1cm,yshift=-1.5cm] (0,4) -- (1,3.5) -- (2,4) -- (1,4.5) -- cycle; 
        \draw [,black,fill=yellow,xshift=2cm,yshift=-3cm] (0,4) -- (1,3.5) -- (2,4) -- (1,4.5) -- cycle; 
        \draw [,black,fill=yellow,xshift=3cm,yshift=-4.5cm] (0,4) -- (1,3.5) -- (2,4) -- (1,4.5) -- cycle; 
    
        \draw [,black,fill=blue] (1,3.5) -- (2,4) -- (2,3) -- (1,2.5) -- cycle;     
        \draw [,black,fill=blue,xshift=1cm,yshift=-1.5cm] (1,3.5) -- (2,4) -- (2,3) -- (1,2.5) -- cycle;  
        \draw [,black,fill=blue,xshift=2cm,yshift=-3cm] (1,3.5) -- (2,4) -- (2,3) -- (1,2.5) -- cycle;  
        \draw [,black,fill=blue,xshift=3cm,yshift=-4.5cm] (1,3.5) -- (2,4) -- (2,3) -- (1,2.5) -- cycle;  
        \end{scope}

        \begin{scope}[shift={(-2.5,-1.25)}]
            \draw [,black,fill=red] (0,0) -- (0,2) -- (1,1.5) -- (1,-0.5) -- cycle;
            \draw[,black] (0,1) -- (1,0.5);
             \draw [,black,fill=green] (1,-0.5) -- (1,0.5) -- (2,0) -- (2,-1) -- cycle;
             \draw [,black,fill=yellow] (0,2) -- (1,2.5) -- (2,2) -- (1,1.5) -- cycle;
             \draw [,black,fill=yellow,xshift=1cm,yshift=-1.5cm] (0,2) -- (1,2.5) -- (2,2) -- (1,1.5) -- cycle;
             \draw [,black,fill=blue] (1,1.5) -- (2,2) -- (2,1) -- (1,0.5) -- cycle; 
             \draw [,black,fill=blue,xshift=1cm,yshift=-1.5cm] (1,1.5) -- (2,2) -- (2,1) -- (1,0.5) -- cycle; 

             \draw [densely dotted,thick] (1,2.5) -- (2.5,3.25);
             \draw [densely dotted,thick,xshift=1cm,yshift=-0.5cm] (1,2.5) -- (2.5,3.25);
             \draw [densely dotted,thick,xshift=1cm,yshift=-1.5cm](1,2.5) -- (2.5,3.25);
             \draw [densely dotted,thick,xshift=2cm,yshift=-2cm](1,2.5) -- (2.5,3.25);
             \draw [densely dotted,thick,xshift=2cm,yshift=-3cm] (1,2.5) -- (2.5,3.25);
        \end{scope}

        \node at (6,1) [] {$\setminus$};

        \begin{scope}[shift={(7.5,0)}]
        \draw [,black,fill=red] (0,0) -- (0,1) -- (4,-1) -- (4,-2) -- cycle;
        \draw [black] (1,-0.5) -- (1,0.5);
        \draw [black] (2,-1) -- (2,0);
        \draw [black] (3,-1.5) -- (3,-0.5);
         \draw [,black,fill=yellow] (0,1) -- (1,0.5) -- (2,1) -- (1,1.5) -- cycle;
         \draw [,black,fill=yellow,xshift=1cm,yshift=-.5cm] (0,1) -- (1,0.5) -- (2,1) -- (1,1.5) -- cycle;
         \draw [,black,fill=yellow,xshift=2cm,yshift=-1cm] (0,1) -- (1,0.5) -- (2,1) -- (1,1.5) -- cycle;
         \draw [,black,fill=yellow,xshift=3cm,yshift=-1.5cm] (0,1) -- (1,0.5) -- (2,1) -- (1,1.5) -- cycle;

         \draw [,black,fill=blue] (4,-1) -- (4,-2) -- (5,-1.5) -- (5,-0.5) -- cycle;
        \end{scope}

        \node at (14,1) [] {$=$};

        \begin{scope}[shift={(18.5,0)}]
        \draw[,black,fill=red] (0,1) -- (0,3) -- (2,2) -- (2,1) -- (3,0.5) -- (3,-0.5) -- cycle; 
        \draw[,black] (1,0.5) -- (1,2.5);
        \draw[,black] (2,0) -- (2,1);
        \draw[,black] (0,1) -- (3,-0.5);
        \draw[,black] (0,2) -- (2,1);

\draw [thin, dashed] (0,0) --(0,1) -- (4,-1) -- (4,-2) -- cycle;
	\draw [thin, dashed] (1,-0.5) -- (1,0.5);
	\draw [thin, dashed] (2,-1) -- (2,0.);
	\draw [thin, dashed] (3,-1.5) -- (3,-0.5);
	\draw [thin, dashed] (4,-2)  -- (5,-1.5) -- (5,-0.5) -- (4,0);
	\draw [thin, dashed] (4,-1) -- (5,-0.5); 
	
        \draw [,black,fill=green] (0,3) -- (0,4) -- (1,3.5) -- (1,2.5) -- cycle; 
        \draw [,black,fill=yellow] (0,4) -- (1,3.5) -- (2,4) -- (1,4.5) -- cycle; 
        \draw [,black,fill=yellow,xshift=1cm,yshift=-1.5cm] (0,4) -- (1,3.5) -- (2,4) -- (1,4.5) -- cycle; 
        \draw [,black,fill=yellow,xshift=2cm,yshift=-3cm] (0,4) -- (1,3.5) -- (2,4) -- (1,4.5) -- cycle; 
    
        \draw [,black,fill=blue] (1,3.5) -- (2,4) -- (2,3) -- (1,2.5) -- cycle;     
        \draw [,black,fill=blue,xshift=1cm,yshift=-1.5cm] (1,3.5) -- (2,4) -- (2,3) -- (1,2.5) -- cycle;  
        \draw [,black,fill=blue,xshift=2cm,yshift=-3cm] (1,3.5) -- (2,4) -- (2,3) -- (1,2.5) -- cycle;  
        \end{scope}

        \begin{scope}[shift={(16.,-1.25)}]
            \draw [,black,fill=red] (0,0) -- (0,2) -- (1,1.5) -- (1,-0.5) -- cycle;
            \draw[,black] (0,1) -- (1,0.5);
             \draw [,black,fill=green] (1,-0.5) -- (1,0.5) -- (2,0) -- (2,-1) -- cycle;
             \draw [,black,fill=yellow] (0,2) -- (1,2.5) -- (2,2) -- (1,1.5) -- cycle;
             \draw [,black,fill=yellow,xshift=1cm,yshift=-1.5cm] (0,2) -- (1,2.5) -- (2,2) -- (1,1.5) -- cycle;
             \draw [,black,fill=blue] (1,1.5) -- (2,2) -- (2,1) -- (1,0.5) -- cycle; 
             \draw [,black,fill=blue,xshift=1cm,yshift=-1.5cm] (1,1.5) -- (2,2) -- (2,1) -- (1,0.5) -- cycle; 

             \draw [densely dotted,thick] (1,2.5) -- (2.5,3.25);
             \draw [densely dotted,thick,xshift=1cm,yshift=-0.5cm] (1,2.5) -- (2.5,3.25);
             \draw [densely dotted,thick,xshift=1cm,yshift=-1.5cm](1,2.5) -- (2.5,3.25);
             \draw [densely dotted,thick,xshift=2cm,yshift=-2cm](1,2.5) -- (2.5,3.25);
             \draw [densely dotted,thick,xshift=2cm,yshift=-3cm] (1,2.5) -- (2.5,3.25);
        \end{scope}

        \node at (24.5,1) [] {$\sim$};

        \begin{scope}[shift={(28,0)}]
        \draw[,black,fill=red] (0,1) -- (0,3) -- (2,2) -- (2,1) -- (3,0.5) -- (3,-0.5) -- cycle; 
        \draw[,black] (1,0.5) -- (1,2.5);
        \draw[,black] (2,0) -- (2,1);
        \draw[,black] (0,1) -- (3,-0.5);
        \draw[,black] (0,2) -- (2,1);

\draw [thin, dashed] (0,0) --(0,1) -- (4,-1) -- (4,-2) -- cycle;
	\draw [thin, dashed] (1,-0.5) -- (1,0.5);
	\draw [thin, dashed] (2,-1) -- (2,0.);
	\draw [thin, dashed] (3,-1.5) -- (3,-0.5);
	\draw [thin, dashed] (4,-2)  -- (5,-1.5) -- (5,-0.5) -- (4,0);
	\draw [thin, dashed] (4,-1) -- (5,-0.5);
	 
        \draw [,black,fill=green] (0,3) -- (0,4) -- (1,3.5) -- (1,2.5) -- cycle; 
        \draw [,black,fill=yellow] (0,4) -- (1,3.5) -- (2,4) -- (1,4.5) -- cycle; 
        \draw [,black,fill=yellow,xshift=1cm,yshift=-1.5cm] (0,4) -- (1,3.5) -- (2,4) -- (1,4.5) -- cycle; 
        \draw [,black,fill=yellow,xshift=2cm,yshift=-3cm] (0,4) -- (1,3.5) -- (2,4) -- (1,4.5) -- cycle; 
    
        \draw [,black,fill=blue] (1,3.5) -- (2,4) -- (2,3) -- (1,2.5) -- cycle;     
        \draw [,black,fill=blue,xshift=1cm,yshift=-1.5cm] (1,3.5) -- (2,4) -- (2,3) -- (1,2.5) -- cycle;  
        \draw [,black,fill=blue,xshift=2cm,yshift=-3cm] (1,3.5) -- (2,4) -- (2,3) -- (1,2.5) -- cycle;  
        \end{scope}

        \begin{scope}[shift={(27,-0.5)}]
            \draw [,black,fill=red] (0,0) -- (0,2) -- (1,1.5) -- (1,-0.5) -- cycle;
            \draw[,black] (0,1) -- (1,0.5);
             \draw [,black,fill=green] (1,-0.5) -- (1,0.5) -- (2,0) -- (2,-1) -- cycle;
             \draw [,black,fill=yellow] (0,2) -- (1,2.5) -- (2,2) -- (1,1.5) -- cycle;
             \draw [,black,fill=yellow,xshift=1cm,yshift=-1.5cm] (0,2) -- (1,2.5) -- (2,2) -- (1,1.5) -- cycle;
             \draw [,black,fill=blue] (1,1.5) -- (2,2) -- (2,1) -- (1,0.5) -- cycle; 
             \draw [,black,fill=blue,xshift=1cm,yshift=-1.5cm] (1,1.5) -- (2,2) -- (2,1) -- (1,0.5) -- cycle; 

        \end{scope}
        
    \end{tikzpicture}

\caption{Construction of the counterexample in $\Quot_2^8 \A^3$ with tangent space of dimension 39.} \label{fig:square-of-maximal-ideal}
\end{figure}

We also mention that this procedure produces counterexamples of rank $r=3$ and lengths $d=10$ and $d=11$. In fact, the modules
\[
M_1=\frac{(x,y,z^2)+I}{I}\qquad\text{and}\qquad M_2=\frac{(x,y,z)+I}{I}
\]
 correspond to points
 \[
 [ 0 \to K_1 \to R^{\oplus 3} \to M_1 \to 0] \in \Quot_3^{10} \A^3 \qquad\text{and}\qquad [ 0 \to K_2 \to R^{\oplus 3} \to M_2 \to 0] \in \Quot_3^{11} \A^3
 \]
 with 
    \[
    \dim_\C \mathsf{T}_{M_1} \Quot_3^{10} \A^3 = 69 \not\equiv 3\cdot10\pmod 2 \qquad\text{and}\qquad \dim_\C \mathsf{T}_{M_2} \Quot_3^{11} \A^3 = 70 \not\equiv 3\cdot11\pmod 2.
    \]
See \Cref{fig:other} for the combinatorial description of these two counterexamples.\footnote[8]{See the ancillary \textit{Macaulay2} \cite{M2} file \href{www.paololella.it/software/other-counterexamples-Quot-scheme.m2}{\tt other-counterexamples-Quot-scheme.m2} for the computation of the tangent space dimensions of these counterexamples.}

\begin{figure}[!ht]
\centering
\begin{tikzpicture}[xscale=0.25,yscale=0.3]
        \begin{scope}[shift={(0,0)}]
        \draw[,black,fill=red] (0,0) -- (0,3) -- (2,2) -- (2,1) -- (3,0.5) -- (3,-0.5) -- (4,-1) -- (4,-2) -- cycle; 
        \draw[,black] (1,-0.5) -- (1,2.5);
        \draw[,black] (2,-1) -- (2,1);
        \draw[,black] (3,-1.5) -- (3,-0.5);
        \draw[,black] (0,1) -- (3,-0.5);
        \draw[,black] (0,2) -- (2,1);

        \draw [,black,fill=green] (0,3) -- (0,4) -- (1,3.5) -- (1,2.5) -- cycle; 
        \draw [,black,fill=yellow] (0,4) -- (1,3.5) -- (2,4) -- (1,4.5) -- cycle; 
        \draw [,black,fill=yellow,xshift=1cm,yshift=-1.5cm] (0,4) -- (1,3.5) -- (2,4) -- (1,4.5) -- cycle; 
        \draw [,black,fill=yellow,xshift=2cm,yshift=-3cm] (0,4) -- (1,3.5) -- (2,4) -- (1,4.5) -- cycle; 
        \draw [,black,fill=yellow,xshift=3cm,yshift=-4.5cm] (0,4) -- (1,3.5) -- (2,4) -- (1,4.5) -- cycle; 
    
        \draw [,black,fill=blue] (1,3.5) -- (2,4) -- (2,3) -- (1,2.5) -- cycle;     
        \draw [,black,fill=blue,xshift=1cm,yshift=-1.5cm] (1,3.5) -- (2,4) -- (2,3) -- (1,2.5) -- cycle;  
        \draw [,black,fill=blue,xshift=2cm,yshift=-3cm] (1,3.5) -- (2,4) -- (2,3) -- (1,2.5) -- cycle;  
        \draw [,black,fill=blue,xshift=3cm,yshift=-4.5cm] (1,3.5) -- (2,4) -- (2,3) -- (1,2.5) -- cycle;  
        \end{scope}

        \begin{scope}[shift={(-2.5,-1.25)}]
            \draw [,black,fill=red] (0,0) -- (0,2) -- (1,1.5) -- (1,-0.5) -- cycle;
            \draw[,black] (0,1) -- (1,0.5);
             \draw [,black,fill=green] (1,-0.5) -- (1,0.5) -- (2,0) -- (2,-1) -- cycle;
             \draw [,black,fill=yellow] (0,2) -- (1,2.5) -- (2,2) -- (1,1.5) -- cycle;
             \draw [,black,fill=yellow,xshift=1cm,yshift=-1.5cm] (0,2) -- (1,2.5) -- (2,2) -- (1,1.5) -- cycle;
             \draw [,black,fill=blue] (1,1.5) -- (2,2) -- (2,1) -- (1,0.5) -- cycle; 
             \draw [,black,fill=blue,xshift=1cm,yshift=-1.5cm] (1,1.5) -- (2,2) -- (2,1) -- (1,0.5) -- cycle; 

             \draw [densely dotted,thick] (1,2.5) -- (2.5,3.25);
             \draw [densely dotted,thick,xshift=1cm,yshift=-0.5cm] (1,2.5) -- (2.5,3.25);
             \draw [densely dotted,thick,xshift=1cm,yshift=-1.5cm](1,2.5) -- (2.5,3.25);
             \draw [densely dotted,thick,xshift=2cm,yshift=-2cm](1,2.5) -- (2.5,3.25);
             \draw [densely dotted,thick,xshift=2cm,yshift=-3cm] (1,2.5) -- (2.5,3.25);
        \end{scope}

        \node at (6,1) [] {$\setminus$};

        \begin{scope}[shift={(7.5,0)}]
        \draw [,black,fill=red] (0,0) -- (0,1) -- (2,0) -- (2,-1) -- cycle;
        \draw [black] (1,-0.5) -- (1,0.5);
         \draw [,black,fill=yellow] (0,1) -- (1,0.5) -- (2,1) -- (1,1.5) -- cycle;
         \draw [,black,fill=yellow,xshift=1cm,yshift=-.5cm] (0,1) -- (1,0.5) -- (2,1) -- (1,1.5) -- cycle;

         \draw [,black,fill=blue] (2,0) -- (2,-1) -- (3,-.5) -- (3,0.5) -- cycle;
        \end{scope}

        \node at (12.5,1) [] {$=$};

        \node at (12.5,-3) [] {$M_1$};

        \begin{scope}[shift={(17.,0)}]
        \draw[,black,fill=red] (0,1) -- (0,3) -- (2,2) -- (2,1) -- (3,0.5) -- (3,-0.5) -- (4,-1) -- (4,-2) -- (2,-1) -- (2,0) -- cycle; 
        \draw[,black] (1,0.5) -- (1,2.5);
        \draw[,black] (2,0) -- (2,1);
        \draw[,black] (3,-1.5) -- (3,-0.5);
        \draw[,black] (0,1) -- (3,-0.5);
        \draw[,black] (0,2) -- (2,1);

\draw[thin,dashed] (0,1) -- (0,0) -- (2,-1);
       \draw[thin,dashed] (1,-0.5) -- (1,0.5);

        \draw [,black,fill=green] (0,3) -- (0,4) -- (1,3.5) -- (1,2.5) -- cycle; 
        \draw [,black,fill=yellow] (0,4) -- (1,3.5) -- (2,4) -- (1,4.5) -- cycle; 
        \draw [,black,fill=yellow,xshift=1cm,yshift=-1.5cm] (0,4) -- (1,3.5) -- (2,4) -- (1,4.5) -- cycle; 
        \draw [,black,fill=yellow,xshift=2cm,yshift=-3cm] (0,4) -- (1,3.5) -- (2,4) -- (1,4.5) -- cycle; 
        \draw [,black,fill=yellow,xshift=3cm,yshift=-4.5cm] (0,4) -- (1,3.5) -- (2,4) -- (1,4.5) -- cycle; 
    
        \draw [,black,fill=blue] (1,3.5) -- (2,4) -- (2,3) -- (1,2.5) -- cycle;     
        \draw [,black,fill=blue,xshift=1cm,yshift=-1.5cm] (1,3.5) -- (2,4) -- (2,3) -- (1,2.5) -- cycle;  
        \draw [,black,fill=blue,xshift=2cm,yshift=-3cm] (1,3.5) -- (2,4) -- (2,3) -- (1,2.5) -- cycle;  
        \draw [,black,fill=blue,xshift=3cm,yshift=-4.5cm] (1,3.5) -- (2,4) -- (2,3) -- (1,2.5) -- cycle;  
        \end{scope}

        \begin{scope}[shift={(14.5,-1.25)}]
            \draw [,black,fill=red] (0,0) -- (0,2) -- (1,1.5) -- (1,-0.5) -- cycle;
            \draw[,black] (0,1) -- (1,0.5);
             \draw [,black,fill=green] (1,-0.5) -- (1,0.5) -- (2,0) -- (2,-1) -- cycle;
             \draw [,black,fill=yellow] (0,2) -- (1,2.5) -- (2,2) -- (1,1.5) -- cycle;
             \draw [,black,fill=yellow,xshift=1cm,yshift=-1.5cm] (0,2) -- (1,2.5) -- (2,2) -- (1,1.5) -- cycle;
             \draw [,black,fill=blue] (1,1.5) -- (2,2) -- (2,1) -- (1,0.5) -- cycle; 
             \draw [,black,fill=blue,xshift=1cm,yshift=-1.5cm] (1,1.5) -- (2,2) -- (2,1) -- (1,0.5) -- cycle; 

             \draw [densely dotted,thick] (1,2.5) -- (2.5,3.25);
             \draw [densely dotted,thick,xshift=1cm,yshift=-0.5cm] (1,2.5) -- (2.5,3.25);
             \draw [densely dotted,thick,xshift=1cm,yshift=-1.5cm](1,2.5) -- (2.5,3.25);
             \draw [densely dotted,thick,xshift=2cm,yshift=-2cm](1,2.5) -- (2.5,3.25);
             \draw [densely dotted,thick,xshift=2cm,yshift=-3cm] (1,2.5) -- (2.5,3.25);
        \end{scope}

    \end{tikzpicture}
    \hspace{2cm}
    \begin{tikzpicture}[xscale=0.25,yscale=0.3]
        \begin{scope}[shift={(0,0)}]
        \draw[,black,fill=red] (0,0) -- (0,3) -- (2,2) -- (2,1) -- (3,0.5) -- (3,-0.5) -- (4,-1) -- (4,-2) -- cycle; 
        \draw[,black] (1,-0.5) -- (1,2.5);
        \draw[,black] (2,-1) -- (2,1);
        \draw[,black] (3,-1.5) -- (3,-0.5);
        \draw[,black] (0,1) -- (3,-0.5);
        \draw[,black] (0,2) -- (2,1);

        \draw [,black,fill=green] (0,3) -- (0,4) -- (1,3.5) -- (1,2.5) -- cycle; 
        \draw [,black,fill=yellow] (0,4) -- (1,3.5) -- (2,4) -- (1,4.5) -- cycle; 
        \draw [,black,fill=yellow,xshift=1cm,yshift=-1.5cm] (0,4) -- (1,3.5) -- (2,4) -- (1,4.5) -- cycle; 
        \draw [,black,fill=yellow,xshift=2cm,yshift=-3cm] (0,4) -- (1,3.5) -- (2,4) -- (1,4.5) -- cycle; 
        \draw [,black,fill=yellow,xshift=3cm,yshift=-4.5cm] (0,4) -- (1,3.5) -- (2,4) -- (1,4.5) -- cycle; 
    
        \draw [,black,fill=blue] (1,3.5) -- (2,4) -- (2,3) -- (1,2.5) -- cycle;     
        \draw [,black,fill=blue,xshift=1cm,yshift=-1.5cm] (1,3.5) -- (2,4) -- (2,3) -- (1,2.5) -- cycle;  
        \draw [,black,fill=blue,xshift=2cm,yshift=-3cm] (1,3.5) -- (2,4) -- (2,3) -- (1,2.5) -- cycle;  
        \draw [,black,fill=blue,xshift=3cm,yshift=-4.5cm] (1,3.5) -- (2,4) -- (2,3) -- (1,2.5) -- cycle;  
        \end{scope}

        \begin{scope}[shift={(-2.5,-1.25)}]
            \draw [,black,fill=red] (0,0) -- (0,2) -- (1,1.5) -- (1,-0.5) -- cycle;
            \draw[,black] (0,1) -- (1,0.5);
             \draw [,black,fill=green] (1,-0.5) -- (1,0.5) -- (2,0) -- (2,-1) -- cycle;
             \draw [,black,fill=yellow] (0,2) -- (1,2.5) -- (2,2) -- (1,1.5) -- cycle;
             \draw [,black,fill=yellow,xshift=1cm,yshift=-1.5cm] (0,2) -- (1,2.5) -- (2,2) -- (1,1.5) -- cycle;
             \draw [,black,fill=blue] (1,1.5) -- (2,2) -- (2,1) -- (1,0.5) -- cycle; 
             \draw [,black,fill=blue,xshift=1cm,yshift=-1.5cm] (1,1.5) -- (2,2) -- (2,1) -- (1,0.5) -- cycle; 

             \draw [densely dotted,thick] (1,2.5) -- (2.5,3.25);
             \draw [densely dotted,thick,xshift=1cm,yshift=-0.5cm] (1,2.5) -- (2.5,3.25);
             \draw [densely dotted,thick,xshift=1cm,yshift=-1.5cm](1,2.5) -- (2.5,3.25);
             \draw [densely dotted,thick,xshift=2cm,yshift=-2cm](1,2.5) -- (2.5,3.25);
             \draw [densely dotted,thick,xshift=2cm,yshift=-3cm] (1,2.5) -- (2.5,3.25);
        \end{scope}

        \node at (6,1) [] {$\setminus$};
\node at (11.5,-3) [] {$M_2$};

        \begin{scope}[shift={(7.5,0)}]
        \draw [,black,fill=red] (0,0) -- (0,1) -- (1,0.5) -- (1,-0.5) -- cycle;
         \draw [,black,fill=yellow] (0,1) -- (1,0.5) -- (2,1) -- (1,1.5) -- cycle;
	
         \draw [,black,fill=blue] (1,0.5) -- (1,-0.5) -- (2,0) -- (2,1) -- cycle;
        \end{scope}

        \node at (11.5,1) [] {$=$};

        \begin{scope}[shift={(16.,0)}]
        \draw[,black,fill=red] (0,1) -- (0,3) -- (2,2) -- (2,1) -- (3,0.5) -- (3,-0.5) -- (4,-1) -- (4,-2) -- (1,-0.5) -- (1,0.5) -- cycle; 
        \draw[,black] (1,0.5) -- (1,2.5);
        \draw[,black] (2,-1) -- (2,1);
        \draw[,black] (3,-1.5) -- (3,-0.5);
        \draw[,black] (0,1) -- (3,-0.5);
        \draw[,black] (0,2) -- (2,1);
\draw[thin,dashed] (0,1) -- (0,0) -- (1,-0.5);

        \draw [,black,fill=green] (0,3) -- (0,4) -- (1,3.5) -- (1,2.5) -- cycle; 
        \draw [,black,fill=yellow] (0,4) -- (1,3.5) -- (2,4) -- (1,4.5) -- cycle; 
        \draw [,black,fill=yellow,xshift=1cm,yshift=-1.5cm] (0,4) -- (1,3.5) -- (2,4) -- (1,4.5) -- cycle; 
        \draw [,black,fill=yellow,xshift=2cm,yshift=-3cm] (0,4) -- (1,3.5) -- (2,4) -- (1,4.5) -- cycle; 
        \draw [,black,fill=yellow,xshift=3cm,yshift=-4.5cm] (0,4) -- (1,3.5) -- (2,4) -- (1,4.5) -- cycle; 
    
        \draw [,black,fill=blue] (1,3.5) -- (2,4) -- (2,3) -- (1,2.5) -- cycle;     
        \draw [,black,fill=blue,xshift=1cm,yshift=-1.5cm] (1,3.5) -- (2,4) -- (2,3) -- (1,2.5) -- cycle;  
        \draw [,black,fill=blue,xshift=2cm,yshift=-3cm] (1,3.5) -- (2,4) -- (2,3) -- (1,2.5) -- cycle;  
        \draw [,black,fill=blue,xshift=3cm,yshift=-4.5cm] (1,3.5) -- (2,4) -- (2,3) -- (1,2.5) -- cycle;  
        \end{scope}

        \begin{scope}[shift={(13.5,-1.25)}]
            \draw [,black,fill=red] (0,0) -- (0,2) -- (1,1.5) -- (1,-0.5) -- cycle;
            \draw[,black] (0,1) -- (1,0.5);
             \draw [,black,fill=green] (1,-0.5) -- (1,0.5) -- (2,0) -- (2,-1) -- cycle;
             \draw [,black,fill=yellow] (0,2) -- (1,2.5) -- (2,2) -- (1,1.5) -- cycle;
             \draw [,black,fill=yellow,xshift=1cm,yshift=-1.5cm] (0,2) -- (1,2.5) -- (2,2) -- (1,1.5) -- cycle;
             \draw [,black,fill=blue] (1,1.5) -- (2,2) -- (2,1) -- (1,0.5) -- cycle; 
             \draw [,black,fill=blue,xshift=1cm,yshift=-1.5cm] (1,1.5) -- (2,2) -- (2,1) -- (1,0.5) -- cycle; 

             \draw [densely dotted,thick] (1,2.5) -- (2.5,3.25);
             \draw [densely dotted,thick,xshift=1cm,yshift=-0.5cm] (1,2.5) -- (2.5,3.25);
             \draw [densely dotted,thick,xshift=1cm,yshift=-1.5cm](1,2.5) -- (2.5,3.25);
             \draw [densely dotted,thick,xshift=2cm,yshift=-2cm](1,2.5) -- (2.5,3.25);
             \draw [densely dotted,thick,xshift=2cm,yshift=-3cm] (1,2.5) -- (2.5,3.25);
        \end{scope}

    \end{tikzpicture}

\caption{Combinatorial construction of counterexamples in $\Quot_3^{10} \A^3$ and $\Quot_3^{11} \A^3$.}\label{fig:other} 
\end{figure}
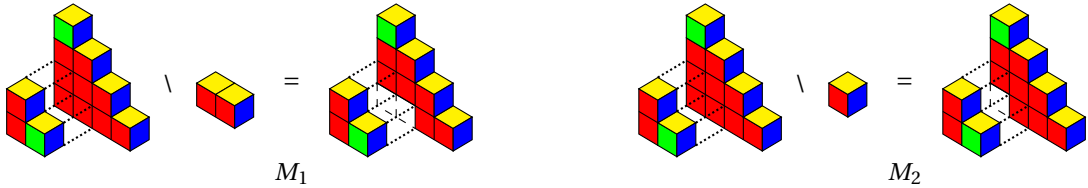

\subsection{A remark on the Behrend function}\label{subsec:bfunction}
The counterexamples given \cite{GGGL} concerns zero-dimensional closed subschemes of $\A^3$ of length 12. This might suggest that the failure of the parity conjecture is related to the irreducibility of the Hilbert schemes of points on smooth threefold. However, the two things seem nowadays to be unrelated as remarked by the authors. Similarly, one can address the same question for the points disproving the constancy of the Behrend function\footnote[9]{The constancy of this function was expected by results in \cite{BFHilb}.} given in \cite{behrendnonconstant}. Notice that these points contradict the parity conjecture as well. Again, the two aspects seem to be unrelated. Indeed, the ideal
\[
J_{JKS}=\left( (x^2)+(y,z)^2\right)^2 +(y^3-x^3z)=\left(x^4, x^2y^2, x^2yz, x^2z^2, y^4, y^3z, y^2z^2, yz^3, y^3-x^3z, z^4 \right)
\]
given in \cite{behrendnonconstant} is smoothable as we show in the following proposition.

\begin{proposition}\label{prop:smoothableBeh}
    The ideal $J_{JKS}$ is smoothable.
 \end{proposition}
\begin{proof}  The existence of the $\A^1$-flat family $\mathcal{Z}\subset \A^3_{x,y,z}\times \A^1_{t}$ defined by the ideal $\mathcal J_{JKS}\subset \C[t][x,y,z]$ given by
     \[ 
    \begin{split}
     \mathcal{J}_{JKS}=&\left( x^4, x^2y^2, xy^3, x^2yz, y^4, x^2z^2+t  x^2z, y^3z+t y^3, y^2z^2-t^2  y^2,yz^3-t^2  yz, x^3z-y^3,z^4-t^2  z^2\right),
     \end{split}
 \]

 certificates the validity of the statement. Indeed, a direct check shows that the general member $Z$ of $\mathcal{Z}$ corresponds to a smooth point of $[Z]\in \Hilb^{24} \A^3$ and that it is the disjoint union of three fat points of length 10, 8, and 6, and hence the smoothability of the scheme defined by the ideal $J_{JKS}$ follows by the irreducibility of $\Hilb^{10} \A^3$ and the fact that smoothability is a closed property. 
 \end{proof} 
 We conclude this section with one example that shows that the Behrend function is not constant on the $\Quot$-schemes as well. This happens, for rank $r$ greater or equal than three, already for $d=13$. 
 
 \begin{example}\label{ex:Beh}
     The Behrend function is not constant on the $\Quot$-schemes of the affine space $\A^3$. This question has been addressed by the third author and A. T. Ricolfi in \cite{GRICOLFI}. This example is constructed following the strategy explained in \Cref{subsec:explanation}.

     Consider the $R$-module $M=J/(J\cap J_{JKS})$, where $J=(y^2,x z,y z) \subset R$. 
     Then, we identify $M$ with a point
     \[
     [0\to K \to R^{\oplus 3} \to M \to 0 ]\in\Quot_3^{13}\A^3
    \]
    with tangent space $\mathsf T_{[M]}\Quot_{3}^{13} \A^3$ of dimension $86$. This is enough to apply the arguments in \cite{behrendnonconstant} and to conclude non-constancy of the Behrend function on $\Quot_{3}^{13}\A^3$. Indeed, as proven in \cite{VIRTUALCOUNTS}, $\Quot$-schemes of $\A^3$ are critical loci.\footnote[10]{See the ancillary \textit{Macaulay2} \cite{M2} file \href{www.paololella.it/software/other-counterexamples-Quot-scheme.m2}{\tt other-counterexamples-Quot-scheme.m2}.}
 \end{example}
 
\section{Nested Hilbert scheme of points}\label{sec:nested}
In this section we mainly study the tangent space to the nested Hilbert scheme. Precisely, in \Cref{subsec:TNT} we review the theory developed in \cite{ELEMENTARY} in nested terms, and we prove \Cref{thmintro:tnt per nested} of the introduction. While, in \Cref{subsec:nonredcomp} we give examples of non-reduced components of the nested Hilbert schemes and we prove \Cref{thmintro:main2}. Finally, in \Cref{subsec:newelem} we provide new examples of elementary components of the Hilbert schemes of points by proving \Cref{thmintro:newcomp}.

\begin{notation}
    In order to ease the notation, for any vector $\mathbf{d}\in \Z^r$ we will denote by $d_i$, for $i=1,\ldots,r$, its entries. Moreover, if $\mathbf{d}\in \Z^r$ is a non-decreasing sequence of non-negative integers, by $\mathbf{d}$-nesting $Z $ in $X$ we mean that $Z=(Z_i)_{i=1}^r$, where $Z_{1}\subset\cdots \subset Z_{r}\subset X$ are closed zero-dimensional subschemes and $\len Z_i=d_i$, for $i=1,\ldots,r$. Finally, the support of the nesting $Z$ is the reduced scheme $\supp Z=\supp Z_r$.
\end{notation}

Let $X$ be a smooth quasi-projective variety and let $\mathbf{d} \in \Z^r$ be a non-decreasing   sequence of non-negative integers. The $\mathbf{d}$-nested Hilbert functor of $X$, $\underline{\Hilb}^{\mathbf{d}}X:\Sch^{\opp}_{\C} \to \Sets$, is the contravariant functor defined as follows
\[
\left(\underline{\Hilb}^{\mathbf{d}}X\right)(S)=\Set{(\Calz_i)_{i=1}^r |\Calz_i\subset X\times S\ 
S\mbox{-flat and $S$-finite closed subschemes},\\   \Calz_i\subset\Calz_{i+1} ,\ \len_S\Calz_i=d_i   }, 
\]
where $\len_S$ denotes the relative length. Analogously to the $\Quot$-functor, $\underline{\Hilb}^{\mathbf{d}}X$ is representable  by a quasi-projective scheme, see \cite[Theorem 4.5.1]{sernesi} and \cite{Kleppe}. We call it \textit{nested Hilbert scheme}\footnote[11]{This scheme is sometimes called flag Hilbert scheme.} and we denote it by $\Hilb^{\mathbf{d}}X$. We will denote points of the nested Hilbert scheme by $[Z]$.

\subsection{Tangent space and negative tangents}\label{subsec:TNT}
In what follows, we generalise some results from \cite{ELEMENTARY} to nested Hilbert schemes.

Let us fix some non-decreasing sequence of non-negative integers $\mathbf d\in \Z^r$  and a point $[Z] \in \Hilb^{\mathbf d} X$. Then, there is a natural identification of the tangent space $\mathsf T_{[Z]} \Hilb^{\mathbf d} X$ with the vector subspace of the direct sum $ \bigoplus_{i=1}^r\mathsf T_{[Z_i]}\Hilb^{d_i} X$ (see \eqref{eq:tangentquot}) consisting of $r$-tuples $(\varphi_i)_{i=1}^r$ such that all the squares of the following diagram 
\begin{equation}\label{eq:tangentvector}
    \begin{tikzcd}[row sep=huge,column sep=huge]
     \Cali_1\arrow[d,"\varphi_1"'] & \Cali_2\arrow[l,hook']\arrow[d,"\varphi_2"]& \Cali_3\arrow[l,hook']\arrow[d,"\varphi_3"]& \Cali_{r-1}\arrow[d,"\varphi_{r-1}"]\arrow[l,dotted,hook']&\Cali_r\arrow[l,hook']\arrow[d,"\varphi_r"]\\
     \Calo_X/\Cali_1 & \Calo_X/\Cali_2\arrow[l,two heads]& \Calo_X/\Cali_3\arrow[l,two heads]& \Calo_X/\Cali_{r-1}\arrow[l,dotted,two heads]&\Calo_X/\Cali_r\arrow[l,two heads],
\end{tikzcd}
\end{equation}
commute {\cite[Section 4.5]{sernesi}}.

Let $X$ be a smooth quasi projective scheme and let $p\in X$ be a closed point with maximal ideal sheaf $\mm_p\subset \Calo_X$. For a fat point $Z\subset X$ supported at $p $ we denote by $\Cali_Z\subset\Calo_X$  the ideal sheaf of $Z$. Moreover, we put 
\[ \Cali_Z^{\geqslant k} = \Cali_Z\cap \mm_p^k\mbox{ and }\Calo_Z^{\geqslant k} = (\mm_p^k + \Cali_Z)/\Cali_Z \subset \Calo_Z. \]

\begin{definition}
Let $ \mathbf{d}\in\Z^r$ be a non-decreasing sequence of non-negative integers. A \emph{fat nesting} in $X$ is a nesting $Z=(Z_i)_{i=1}^r$ of fat points in $X$ with same support $p\in X$, i.e.~$[Z]\in \Hilb^{\mathbf d} X$ and
$\sqrt{\Cali_{Z_i}}=\mm_p$, where $\mm_p\subset \Calo_X$ is the maximal ideal sheaf of $p$, for all $i=1,\ldots,r$.  

    Moreover, an irreducible component $V\subset \Hilb^{\mathbf d} X$ is \textit{elementary} if it parametrises just fat nestings, and \textit{composite} otherwise. Finally, a nesting  $Z$ in $X$ corresponding to a point of a composite component is said to be \textit{cleavable} (cf. \cite{Iarrocomponent,ELEMENTARY}).
    
\end{definition}
\begin{definition}\label{def:negativetangents}
Let $\mathbf{d}\in \Z^r$ be a non-decreasing sequence of non-negative integers and let $[Z]\in \Hilb^{\mathbf d} X$ be a fat nesting.
Then, the \textit{non-negative part of the tangent space} $\mathsf T_{ [Z]}   \Hilb^{\mathbf d} X$  is the following vector subspace
\[
\mathsf T_{ [Z]}^{\geqslant0}  \Hilb^{\mathbf d} X  = \Set{\varphi\in\mathsf T_{ [Z]} \Hilb^{\mathbf d} X | \varphi(\Cali_{Z_i}^{\geqslant k}) \subset \Calo_{Z_i}^{\geqslant k}\mbox{ for all }k \in\N\mbox{ and for }i=1,\ldots,r}. 
\]
While, the \textit{negative tangent space} at $[Z]\in \Hilb^{\mathbf d} X$ is
\[
\mathsf T_{ [Z]}^{< 0}  \Hilb^{\mathbf d} X=\frac{ \mathsf T_{ [Z]}  \Hilb^{\mathbf d} X}{ \mathsf T_{ [Z]}^{\geqslant0}  \Hilb^{\mathbf d} X}.
\]
\end{definition}
The\hfill non-negative\hfill part\hfill of\hfill the\hfill tangent\hfill space\hfill can\hfill be\hfill interpreted\hfill as\hfill the\hfill tangent\hfill space\hfill to\hfill the\hfill so-called\\ \textit{Bia{\l{}}ynicki-Birula cell},
whose definition we recall now. Let $Z$ be a fat nesting and consider the diagonal action of the torus $\mathbb G_m=\Spec \C[s,s^{-1}]$ on $\mathsf T_{[Z]}\Hilb^{\mathbf d} X $ given by homotheties. Given a point $p\in X$,  the corresponding Bia{\l{}}ynicki-Birula cell is the quasi-projective scheme $\Hilb_p^{\mathbf{d}} X$ representing the following functor
\[ 
   \left( \underline{\Hilb}_p^{\mathbf d,+}X\right)(S) = \Set{\varphi\colon \overline{\mathbb G}_m \times S \rightarrow \Hilb^{\mathbf d} X | \supp Z=p,\ \forall\ [Z]\in  \varphi(\overline{\mathbb G}_m \times S ),\  \varphi \mbox{ is $\mathbb G_m$-equivariant}} 
\]
where, by convention $\overline{\mathbb G}_m= \Spec \C[ s^{-1}]$.

Set-theoretically it is the subset of the nested Hilbert scheme parametrising fat nestings supported at $p$, i.e.
\[
\Hilb_p^{\mathbf d} X=\Set{[Z]\in \Hilb^{\mathbf d} X| \supp (Z_i)=p, \mbox{ for all }i=1,\ldots,r}.
\] 
The deformation theory of a point $[Z]\in \Hilb_p^{\mathbf d} X$ is described in terms of the tangent space   $\mathsf T_{ [Z]}  \Hilb_p^{\mathbf d} X$. The following proposition from \cite{ELEMENTARY} expresses it in  terms of the non-negative tangent space at $[Z]$.
\begin{proposition}[{\cite[Theorem 4.11]{ELEMENTARY}}]\label{rem:nonneg-punctual}
Let $[Z]\in \Hilb_p^{\mathbf d} X$ be a fat nesting. Then, we have
    \[
    \mathsf T_{ [Z]}  \Hilb_p^{\mathbf d}= \mathsf T_{ [Z]}^{\geqslant0}  \Hilb^{\mathbf d} X .
    \]    
\end{proposition}

    Notice that non-negative tangent vectors can be understood as concatenation of commutative diagrams of the form \eqref{eq:tangentvector} where $\varphi_i\in \mathsf T_{[Z_i]}^{\ge0} \Hilb^{d_i} X$, for all $i=1,\ldots,r$.
\begin{remark}\label{rem:deftheta}
  As shown in \cite{ELEMENTARY}, when $[Z]\in \Hilb^{d} X$ is a fat point  the tangent space of $X$ at $\supp Z$ maps to the tangent space to $ \Hilb^{d} X$ at $[Z]$. Similarly this happens for fat nestings and we give now some details. First we can suppose that $X=\A^nC$ and $\supp Z=\Set{0}$, see \Cref{rem:etale}. We denote by $R=\C[x_1,\ldots,x_n]$ the polynomial ring in $n$ variables. Then, one can identify partial derivatives $\frac{\partial}{\partial x_i}$, for $i=1,\ldots,n$, with a basis of the tangent space $\mathsf T_0\A^n$.
  This naturally induces a map 
  \[
    \begin{tikzcd}
         \mathsf T_{0} \A^n \arrow[r,"\widetilde\theta"] & \mathsf T_{[Z]}\Hilb^{\mathbf d} \A^n,
    \end{tikzcd}
  \]
associating tangent vectors to $\A^n$ at the origin to deformations consisting of translations. More precisely, the partial derivative $\frac{\partial}{\partial x_j}$, for $j=1,\ldots,n$, maps to  translations of all the schemes $Z_i$, for $i=1,\ldots,r$, along the $j$-th coordinate axis preserving the nesting conditions.

  We denote by $\theta$ the map $\theta : \mathsf T_{0} \A^n \to \mathsf T_{ [Z]}^{< 0} \Hilb^{\mathbf d} \A^n
  $ defined as  the composition of $\widetilde \theta$ with the projection defining the negative tangent space, see \Cref{def:negativetangents}.
\end{remark}
\begin{definition}
    Let $[Z]\in \Hilb^{\mathbf d} X$ be a fat nesting. Then, $[Z]$ has TNT (Trivial Negative Tangents) if the   map 
    \[
    \begin{tikzcd}
         \mathsf T_{\supp Z} X \arrow[r,"\theta"] &\mathsf T_{ [Z]}^{< 0}  \Hilb^{\mathbf d} X
    \end{tikzcd}
    \]
    is surjective.
\end{definition} 
We move now to the proof of \Cref{thmintro:tnt per nested} of the introduction which is a generalisation of {\cite[Theorem 4.9]{ELEMENTARY}}. 
\begin{theorem} \label{thm:tnt per nested} 
Let $\mathbf{d}\in \Z^r$ be any non-decreasing sequence of non-negative integers and let $V\subset \Hilb^{\mathbf d} X$ be an irreducible component.  Suppose that $V$ is generically reduced. Then $V$ is elementary if and only if a general point of $V$ has trivial negative tangents.  
\end{theorem}
\begin{proof}
Let $V\subset \Hilb^{\mathbf d} X$ be a generically reduced irreducible  component and let $[Z]\in V$  be a point having TNT. Notice that  $Z$ must be a fat nesting by definition of TNT. Moreover, $V$ is elementary because the only negative tangents are the translations (see \Cref{rem:nonneg-punctual}).

Viceversa, fix a generically reduced elementary component $V\subset  \Hilb^{\mathbf d} X$. Then, étale locally around a general point $[Z]\in V$ the nested Hilbert scheme is isomorphic to $\Hilb_{\supp Z}^{\mathbf{d}}X\times X$ via the association that  forgets about the limit and translates the support. This can be understood at the tangent space level.
Precisely, the natural map
\[
\begin{tikzcd}
    \mathsf T_{[Z]}\Hilb_{\supp Z}^{\mathbf{d}}X\oplus \mathsf T_{\supp Z} X\arrow[r]& \mathsf T_{[Z]}\Hilb^{\mathbf d} X
\end{tikzcd}
\]
 is generically injective on $V$ by \cite[Corollary 4.7]{ELEMENTARY}  and it is surjective because $V$ and $\Hilb_{\supp Z}^{\mathbf{d}}X\times X$ have the same dimension and are generically reduced and hence smooth. As a consequence, the generic point of $V$ has TNT as the generic point of $\Hilb_{\supp Z}^{\mathbf{d}}X\times X$ has. 
\end{proof}
Thanks to \Cref{thm:tnt per nested}, in the next subsection we exhibit some generically non-reduced components of the nested Hilbert scheme of points.
\subsection{Some elementary generically non-reduced components}\label{subsec:nonredcomp}
We prove now \Cref{thmintro:main2} and \Cref{thmintro:solvereduced} of the introduction.
\begin{theorem}\label{thm:main2}
    Let $X$ be a smooth quasi-projective variety and let $d\geqslant 2$ be a positive integer. Let $V\subset \Hilb^d X$ be a generically reduced elementary component. Then, the nested Hilbert scheme $\Hilb^{(1,d)} X$ has a generically non-reduced elementary component $\widetilde{V}$ such that
    \[
    \widetilde{V}_{\redu} \cong V_{\redu}.
    \]
\end{theorem}
\begin{proof} 

    Set theoretically, we have
    \begin{equation}\label{eq:Vtilde}
    \widetilde V =\Set{ [(\supp{Z},Z) ]\in \Hilb^{(1,d)} X | [Z]\in V  },
    \end{equation}
which clearly defines an elementary component since cleavability of the fat nesting  would imply cleavability of $Z$.
We show that the general point $[(p,Z)]\in \widetilde V$ has not TNT. This is enough to conclude thanks to \Cref{thm:tnt per nested}. Notice that, for  $[(p,Z)]\in\widetilde V$ general in $\Hilb^{(1,d)} X$, the point $[Z]\in  V$ has TNT because $V$ is elementary and generically reduced (\cite[Theorem 4.9]{ELEMENTARY}).
Notice also that $[p]\in \Hilb^1 X$ has TNT because $X$ is smooth and hence reduced. 

We argue as in \Cref{rem:etale} and we put  $X=\A^n$ and $p=0\in\A^n$. Let us denote by $I$ the ideal of the subscheme $Z$ and by $\mm =(x_1,\ldots,x_n)$ the (maximal) ideal of the origin. In particular, we have $\mathfrak m=\sqrt{ I}$. Let us also denote by $\iota\colon I \to \mm$ the inclusion and, by $\pi\colon R/I \to R/\mm$ the canonical projection induced by $\iota$.   Now we focus on the negative tangents. Under our assumption, a negative tangent vector to $\widetilde V$ is a diagram of the following form
\[
\begin{tikzcd}[row sep=huge,column sep=huge]
     I \arrow[r,"\frac{\partial}{\partial x_i}" ]\arrow[d,hook',"\iota"']&R/I\arrow[d,two heads,"\pi"]\\
      \mm \arrow[r,"\frac{\partial}{\partial x_j}"' ] & R/\mm,
\end{tikzcd}
\]
 for some $i,j\in\Set{1,\ldots,n}$.
Now, if  $Z$ has embedding dimension $n$ each of the generators of $I$ have zero linear part. As a consequence the diagram commutes because both the compositions $\frac{\partial}{\partial x_j}\circ \iota$ and $\pi\circ\frac{\partial}{\partial x_i}$ are zero. When $Z$ has embedding dimension $m\le n $, up to étale cover, one can suppose that the ideal $I$ is generated by $n-m$ linear forms and generators having zero linear part. To conclude, in both cases the map $\theta$ cannot be surjective for dimensional reasons.
\end{proof}
\begin{corollary}\label{cor:tomain2}
    Let $X$ be a smooth quasi projective variety. Then, for $d\geqslant 2$ all the elementary components of $\Hilb^{(1,d)} X$ are generically non-reduced. 
\end{corollary} 
\begin{remark}
   The statements of \Cref{thm:main2} and \Cref{cor:tomain2} are trivial in the following cases:
\begin{itemize}
    \item $\dim X \leqslant 2$,
    \item $\dim X=3$ and $d \leqslant 11$,
    \item $\dim X \geqslant 4$ and $d \leqslant 7$.
\end{itemize}
This is true because in the above mentioned cases the Hilbert scheme of points is irreducible.
\end{remark}
\begin{example}\label{ex:finalenested}
It is well known that the Hilbert scheme of eight points in $\A^4$ is reducible and it consists of two components
\[
\Hilb^8 \A^4 =H\cup V.
\]
Since $\Hilb^7 \A^4 $ is irreducible, one among $H$ and $V$ must be elementary.  Here we denote by $H$ the smoothable component and by $V$ the elementary one.
An example of non-smoothable point is provided by the ideal $[I]\in V$ defined by
\[
I=(x,y)^2+(y,z)^2+(xz-yw) \subseteq \C[x,y,z,w],
\]
see \cite{8POINTS}. The component $V$ is known to be generically reduced. Precisely, we have $V\cong\Gr(7,10)\times \A^4$. Hence, the tangent space at the generic point of $V$ is 25-dimensional.

    Consider now the elementary component $\widetilde V \subset \Hilb^{(1,8)} \A^4 $ defined as in \eqref{eq:Vtilde}.
    Let us denote by $\mathfrak m$ in $\C[x,y,z,w]$ the maximal ideal of the origin $0\in\A^4$. Then, we have $[\mathfrak m\supset I]\in \widetilde V$. And a direct computation shows that
    \[
    \dim_\C \mathsf T_{[\mathfrak m\supset I]}\widetilde V=\dim_\C \mathsf T_{[\mathfrak m\supset I]} \Hilb^{(1,8)} \A^4=29.
    \]
\end{example}

Although we are able to determine when the first generically   non-reduced component of $X^{[\mathbf d]}$ arises only for $\dim X\ge 4$, we can still say something about the lower dimension cases. Indeed, \Cref{thm:solvereduced} answers negatively to the open question regarding the reducedness of the nested Hilbert scheme of points on surfaces.
It is though worth mentioning that, because of its applications the reducedness problem concerns only nestings of length one  in its original formulation, see \cite[Problem X]{JoachimQuestions}. Here, we prove the existence of generically non-reduced components. The problem of determining an explicit component and the minimal length of a nesting producing generically non-reduced components remains open.

\begin{theorem}\label{thm:solvereduced}
    Let $S $ be a smooth quasi-projective surface. Then, there exists a non-decreasing sequence of non-negative integers $\mathbf d\in\mathbb Z^r$ such that $S^{[\mathbf{d}]}$ has a generically non-reduced component. 
\end{theorem}
\begin{proof}
    By the results in \cite{ALESSIONESTED}, we know that the nested Hilbert scheme of points on surfaces is in general reducible. Hence, there exists a  non-decreasing sequence of non-negative integers $\mathbf d\in\mathbb Z^r$ such that $S^{[\mathbf{d}]}$ has an elementary non-standard component $V\subset S^{[\mathbf{d}]}$. Suppose that $V$ is generically reduced.

    Without loss of generality, we can suppose $d_1>1$. Indeed, if $d_1=1$ then also $S^{[(d_2,\ldots,d_r)]}$ has an elementary component.
    
    Let us now consider the non-decreasing vector $\widetilde{\mathbf{d}}=(1,d_1,\ldots,d_r)\in\mathbb Z^{r+1}$. Define $\widetilde V\subset S^{[\widetilde{\mathbf{d}}]}$ to be the locus
    \[
    \widetilde V=\Set{[Z_0\subset \cdots\subset Z_r] \in S^{[\widetilde{\mathbf{d}}]} | [Z_1 \subset \cdots \subset  Z_r]\in V  }.
    \]
   Then, $\widetilde V$ is an elementary and generically non-reduced irreducible component, as follows by the proof of \Cref{thm:main2}.
\end{proof}

\subsection{Elementary components of classical Hilbert schemes of points}\label{subsec:newelem}
The component $V$ in \Cref{ex:finalenested} was introduced in \cite{Iarrob} and then generalised firstly in \cite{ELEMENTARY} and secondly in \cite{Satrianostaal}. Other examples of elementary components of Hilbert schemes of points can be found in \cite{SomeElementary,MoreElementary,Sha90}. We conclude this paper by proposing an alternative generalisation which gives a new class of elementary components of the Hilbert schemes of points.
\begin{theorem}\label{prop:newcomp}
    Let $R=\C[x_1,\ldots,x_n,y_1,\ldots,y_n]$, for $n\geqslant 2$, be the polynomial ring in $2n$ variables and complex coefficients.
    Then, the ideal
    \begin{equation}
        \label{eq:elementary} I=\sum_{i=1}^n (x_i,y_i)^2+ (x_1\cdots x_n -y_1\cdots y_n),
    \end{equation}
    has TNT. Therefore, by \cite[Theorem 4.5]{ELEMENTARY} every component of the Hilbert scheme $\Hilb^{3^n-1} \A^{2n}$ containing it is elementary.
\end{theorem}
\begin{proof}
In the proof we will denote by $f\in I$ the generator $f=x_1\cdots x_n -y_1\cdots y_n$. Moreover, we denote by $x_{\widehat i}$ the product 
\[
x_{\widehat i}=\prod_{j\not=i}x_j=\frac{\partial f}{\partial x_i}, \quad y_{\widehat i}=\prod_{j\not=i}y_j=\frac{\partial f}{\partial y_i}.
\]

First notice that the ideal $\sum_{i=1}^n (x_i,y_i)^2$ has colength $3^n$ because it cuts out the product of $n$ zero-dimensional schemes of length $3$. As a consequence, we have $[I]\in \Hilb^{3^n-1} \A^{2n}$.

Now observe that the socle of $I$ is concentrated in degree $n$. Indeed, a monomial basis $\mathcal{N}$ for the quotient $R/I$ is given by
\begin{equation}
    \label{eq:socnewcomp}
    \mathcal{N}=\Set{\prod_{j\in A} x_j \cdot \prod_{j\in B}y_j | A,B\subset \Set{1,\ldots,n} \mbox{ and }A\cap B=\emptyset,}
\end{equation}
and a monomial belongs to the socle if and only if all the indices $i=1,\ldots,n$ appear in \eqref{eq:socnewcomp}, i.e. if and only if $A\cup B=\Set{1,\ldots,n}$. As a consequence a monomial  in $\mathcal{N}$ is a socle element if and only if it has degree $n$.

    The ideal $I$ has $4n$ syzygies. They are of two kinds:
    \begin{enumerate}
        \item\label{it:syz1} $y_i\cdot\left( x_i^{2-k}y_i^k \right)-x_i\cdot\left(x_i^{1-k}y_i^{k+1}\right)$, for $i=1,\ldots,n$ and $k=0,1$,
        \item\label{it:syz2} $x_i^{1-k}y_i^k\cdot f-x_{\widehat i}\cdot\left(x_i^{2-k}y_i^{k}\right)+y_{\widehat i}\cdot\left(x_i^{1-k}y_i^{k+1}\right)$, for $i=1,\ldots,n$ and $k=0,1$.
    \end{enumerate} 

    The ideal $I$ is homogeneous, hence $\mathbb G_m$-invariant, see \Cref{subsec:TNT}. In this setting, the $\mathbb G_m$-action lifts to the tangent space at $[I]$. This induces a decomposition of  $\mathsf T_{[I]} \Hilb^{3^n-1} \A^{2n}$  as a direct sum
    \begin{equation}
    \label{eq:splttorus}
            \mathsf T_{[I]} \Hilb^{3^n-1} \A^{2n}=\bigoplus_{k\in\Z} \mathsf T_{[I]}^{=k} \Hilb^{3^n-1} \A^{2n},
    \end{equation}
    where $\mathbb G_m$ acts on $\mathsf T_{[I]}^{=k} \Hilb^{3^n-1} \A^{2n}$ with weight $k$. Now, the negative tangent space naturally identifies with the direct sum of the negative parts in \eqref{eq:splttorus}, i.e.
    \[
    \mathsf T_{[I]}^{<0}\Hilb^{3^n-1} \A^{2n}=\bigoplus_{k<0} \mathsf T_{[I]}^{=k} \Hilb^{3^n-1} \A^{2n},
    \]
    see \cite[Section 2]{ELEMENTARY}. In order to prove the statement we first show that 
    \[
    \bigoplus_{k<-1} \mathsf T_{[I]}^{=k} \Hilb^{3^n-1} \A^{2n}=0,
    \]
    and then we show that 
    \[
    \dim_\C  \mathsf T_{[I]}^{=-1} \Hilb^{3^n-1} \A^{2n}=2n.
    \]
    This is enough to conclude thanks to \cite[Corollary 4.7]{ELEMENTARY}.
    Let $b\in I $ be a generator of degree two and let $\varphi\in \bigoplus_{k<-1} \mathsf T_{[I]}^{=k} \Hilb^{3^n-1} \A^{2n}$ be a tangent vector. Then, the syzygies in (\ref{it:syz1}) implies that $\varphi(b)=0$ and this, together with syzygies (\ref{it:syz2}), implies that $\varphi(f)$ belongs to the socle of $I$. As we already explained the socle is concentrated in degree $n$ and the degree of $\varphi$ implies $\varphi(f)=0$. As a consequence $\varphi\equiv0$.

    Let us now consider some tangent vector  $\varphi\in  \mathsf T_{[I]}^{=-1} \Hilb^{3^n-1} \A^{2n} $  of weight $-1$. The syzygies (\ref{it:syz1}) and (\ref{it:syz2}) imply that
    \begin{enumerate}[\it (i)]
        \item\label{it:fin1} $\varphi(x_iy_i)\in\Span_\C(x_i,y_i)$, for $i=1,\ldots,n$,
        \item\label{it:fin2} $\varphi(x_i^2)\in\Span_\C(x_i)$, for $i=1,\ldots,n$,
        \item\label{it:fin3} $\varphi(y_i^2)\in\Span_\C(y_i)$, for $i=1,\ldots,n$,
        \item\label{it:fin4} $\varphi(f)\in\Span_\C\left(\frac{\partial f}{\partial x_i},\frac{\partial f}{\partial y_i}\ \middle\vert\ i=1,\ldots,n\right)$.
    \end{enumerate}

Let now $W_I\subset R$ be the finite-dimensional complex vector subspace generated by the minimal set of generators of $I$ given in \eqref{eq:elementary}. Let us denote by $H_I$ the following vector subspace
\[
H_I=\Set{ \varphi\in \Hom_\C(W_I,R/I)| \textit{(\ref{it:fin1})},\ldots,\textit{(\ref{it:fin4})}  }\subset \Hom_\C(W_I,R/I).
\]
Then, we have a natural inclusion
\[
\Hom_R(I,R/I)\subset H_I,
\]
coming from a forgetful functor and $\dim_\C H_I=6n$.
To conclude, notice that each syzygies of the same type (\ref{it:syz1}) or (\ref{it:syz2}) impose independent conditions on independent generators. Therefore, each set of syzygies imposes $2n$ conditions. It remains to observe that there is no relation between type (\ref{it:syz1}) and (\ref{it:syz2}). This is true for degree reasons, indeed the ideal $I$ is homogeneous and hence its syzygies are homogeneous as well and, having different degrees they are independent.

As a consequence, the vector subspace $\Hom_R(I,R/I)\subset H_I$ has complex dimension $6n-4n=2n$.
\end{proof}
\begin{remark}
    \Cref{prop:newcomp}, together with the techniques in \cite{Satrianostaal} suggests a possible way to find elementary components of the Hilbert schemes of points. Precisely, first one can construct an elementary component $V\subset \Hilb^d \A^n$ for some $n,d\ge 1$ using \Cref{prop:newcomp} and then by arguing as in \cite{Satrianostaal}  add socle elements to the ideals in $V$ in order to get elementary components for smaller $d$. Moreover, looking at obstruction spaces one can also study the smoothness of the point $[I]$ of the Hilbert scheme. We leave this analysis for future research.
\end{remark}


\end{document}